\documentclass[11pt]{amsart}
\usepackage{verbatim, latexsym, amssymb, amsmath,color}
\usepackage{hyperref} 
\usepackage{epsfig}
\usepackage{bm}
\usepackage{mathrsfs}
\usepackage[margin=1.25in]{geometry}

\newtheorem{theorem}{Theorem}[section]
\newtheorem{corollary}[theorem]{Corollary}
\newtheorem{claim}[]{Claim}
\newtheorem{lemma}[theorem]{Lemma}

\newtheorem{proposition}[theorem]{Proposition}

\theoremstyle{definition}
\newtheorem{definition}[theorem]{Definition}

\newtheorem{remark}[theorem]{Remark}

\newtheorem*{acknowledgements}{Acknowledgements}

\numberwithin{equation}{section}

\newcommand{\al}{\alpha}

\newcommand{\V}{\mathcal{V}}

\newcommand{\R}{\mathbb{R}}
\newcommand{\N}{\mathbb{N}}
\newcommand{\mH}{\mathcal{H}}
\newcommand{\F}{\mathcal{F}}

\newcommand{\A}{\mathcal{A}}

\newcommand{\de}{\delta}

\newcommand{\mZ}{\mathbb{Z}}
\newcommand{\Z}{\mathcal{Z}}
\newcommand{\mS}{\mathcal{S}}

\newcommand{\M}{\mathbf{M}}

\newcommand{\bL}{\mathbf{L}}

\newcommand{\mf}{\mathbf{f}}

\newcommand{\mF}{\mathbf{F}}
\newcommand{\mI}{\mathbf{I}}

\newcommand{\tB}{\widetilde{B}}
\newcommand{\tC}{\widetilde{C}}

\newcommand{\spt}{\operatorname{spt}}
\newcommand{\dist}{\operatorname{dist}}

\newcommand{\inj}{\operatorname{inj}}

\newcommand{\interior}{\operatorname{int}}
\newcommand{\Ric}{\operatorname{Ric}}
\newcommand{\Clos}{\operatorname{Clos}}

\newcommand{\rom}[1]{\expandafter\romannumeral #1}

\newcommand{\an}{\textnormal{An}^G}

\title[Equivariant min-max hypersurface in $G$-manifolds with positive Ricci curvature]{Equivariant min-max hypersurface in $G$-manifolds with positive Ricci curvature}
\author{Tongrui Wang}
\address{Institute for Theoretical Sciences, Westlake Institute for Advanced Study, Westlake University, Hangzhou, Zhejiang, 310024, China}
\email{wangtongrui@westlake.edu.cn}

\begin{document}
\maketitle
\begin{abstract}
	In this paper, we consider a connected orientable closed Riemannian manifold $M^{n+1}$ with positive Ricci curvature. 
	Suppose $G$ is a compact Lie group acting by isometries on $M$ with $3\leq {\rm codim}(G\cdot p)\leq 7$ for all $p\in M$. 
	Then we show the equivariant min-max $G$-hypersurface $\Sigma$ corresponding to the fundamental class $[M]$ is a multiplicity one minimal $G$-hypersurface with a $G$-invariant unit normal and $G$-equivariant index one. 
	As an application, we are able to establish a genus bound for $\Sigma$, a control on the singular points of $\Sigma/G$, and an upper bound for the (first) $G$-width of $M$ provided $n+1=3$ and the actions of $G$ are orientation preserving. 
\end{abstract}

\section{Introduction}

Given a connected orientable closed Riemannian manifold $(M^{n+1}, g_{_M})$, minimizing the area within a non-trivial homology class is a natural way to construct minimal hypersurfaces (c.f. \cite{federer2014geometric}\cite{simon1983lectures}). 
However, if $M$ has positive Ricci curvature, it follows from the stability inequality that this minimization method can not be applied. 
In the 1960s, Almgren \cite{almgren1962homotopy}\cite{almgren1965theory} proposed the {\em min-max theory} to find minimal submanifolds in the most general situation. 
Subsequently, 
the regularity for min-max hypersurfaces was improved by Pitts \cite{pitts2014existence} ($n\leq 5$) and Schoen-Simon \cite{schoen1981regularity} ($n= 6$). 
Indeed, for $n\geq 7$, they showed the min-max minimal hypersurface is smooth embedded except for a singular set of codimension $7$. 

Due to the generality and abstractness of Almgren-Pitts min-max theory, many of the geometric properties of min-max hypersurfaces 
have not been understood until recently. 
For instance, in a closed manifold with positive Ricci curvature, a series of studies were set out to characterize the min-max hypersurfaces generated from one-parameter families. 
Specifically, using the Heegaard splitting, Marques-Neves \cite{marques2012rigidity} studied the index and genus of the min-max surface in certain $3$-manifolds. 
Additionally, they also obtained sharp estimates for the width and rigidity results. 
In higher dimensional manifold $M^{n+1}$ with positive Ricci curvature, Zhou determined the Morse index and multiplicity of the min-max hypersurface in \cite{zhou2015min} (for $3\leq n+1\leq 7$) and \cite{zhou2017min} (for $n\geq 7$). 
Subsequently, Ketover-Marques-Neves \cite{ketover2020catenoid} refined Zhou's results in dimension $3\leq n+1\leq 7$ by showing the orientability of the min-max hypersurface using the catenoid estimates. 
In particular, the min-max hypersurface is an orientable closed minimal hypersurface of Morse index one and has the least area among all orientable closed minimal hypersurfaces. 
Furthermore, without any curvature assumption, the constructions in \cite{marques2012rigidity}\cite{zhou2015min} were also employed by Mazet-Rosenberg \cite{mazet2017minimal} to show the least area minimal hypersurface is either stable or a min-max hypersurface of Morse index one. 

Given a $3$-manifold $M$ with a finite group $G$ acting by isometries, Pitts-Rubinstein \cite{pitts1987applications}\cite{pitts1988equivariant} first assert the existence of a $G$-invariant minimal surface with estimates on its index and genus. 
The existence and regularity for minimal $G$-invariant surfaces (abbreviated as {\em $G$-surfaces}) were recently confirmed by Ketover \cite{ketover2016equivariant} using the equivariant min-max under the smooth setting. 
More generally, suppose $M^{n+1}$ is a closed Riemannian manifold with a compact Lie group $G$ acting by isometries so that $3\leq {\rm codim}(G\cdot p)\leq 7$, $\forall p\in M$. 
The equivariant min-max theory was also extended to this general scenario by Liu \cite{liu2021existence} in the smooth setting and by the author \cite{wang2022min}\cite{wang2022free} in the Almgren-Pitts setting. 
In particular, the author showed an isomorphism between $H_{n+1}(M;\mZ_2)$ and $\pi_1(\Z_n^G(M;\mZ_2) )$ in \cite[Theorem 9]{wang2022min}, where $\Z_n^G(M;\mZ_2)$ is the space of $G$-invariant $n$-cycles. 
Then it parallels the constructions of Almgren-Pitts \cite{pitts2014existence} that the fundamental class $[M] \in H_{n+1}(M;\mZ_2)$ is corresponding to the {\em equivariant min-max width} $W^G(M)>0$ of $M$, which can be realized by the area of some minimal $G$-invariant hypersurfaces (abbreviated as {\em $G$-hypersurfaces}) with multiplicities. 
Therefore, it now seems reasonable to investigate the geometric features of the equivariant min-max hypersurface, such as its area, multiplicity, index, and topology.

In this paper, our main result generalizes the characterization of the min-max hypersurface into an equivariant version (see Theorem \ref{Thm: main 1}). 

\begin{theorem}\label{Thm: main 1 sketch}
	Let $(M^{n+1}, g_{_M})$ be a connected orientable closed Riemannian manifold with positive Ricci curvature, and $G$ be a compact Lie group acting by isometries on $M$ so that $3\leq {\rm codim}(G\cdot p)\leq 7$ for all $p\in M$. 
	Then the equivariant min-max hypersurface $\Sigma$ corresponding to the fundamental class $[M]\in H_{n+1}(M;\mZ_2)$ is a multiplicity one minimal $G$-hypersurface so that 
	\begin{itemize}
		\item[(i)] $\Sigma$ has a $G$-invariant unit normal vector field;
		\item[(ii)] the equivariant Morse index of $\Sigma$ (Definition \ref{Def: index and G-index}) is one;
		\item[(iii)] $\Sigma$ has the least area among all closed embedded minimal $G$-hypersurfaces with $G$-invariant unit normal vector fields. 
	\end{itemize}
\end{theorem}
\begin{remark} 
	We make some remarks for the above theorem:
	\begin{itemize}
		\item[(i)] If $M$ has connected components $\{M_i\}_{i=1}^m$, we can take a component $M_i$ and the Lie sub-group $G_i:=\{g\in G: g\cdot M_i=M_i\}$. 
			By applying the above theorem to $M_i$ and $G_i$, we obtain a minimal $G_i$-invariant hypersurface $\Sigma_i$ of multiplicity one. 
			Additionally, one easily verifies that $G\cdot \Sigma_i\subset G\cdot M_i$ is a minimal $G$-hypersurface satisfying (i)-(iii) in Theorem \ref{Thm: main 1 sketch} with $G\cdot M_i$ in place of $M$. 
		\item[(ii)] Without the positive Ricci curvature assumption, we can combine the proof of Theorem \ref{Thm: main 1 sketch} and the constructions in \cite{mazet2017minimal} to show the existence of a minimal $G$-hypersurface of the least area (counted with multiplicity) among all minimal $G$-hypersurfaces. 
			The details will be discussed in another upcoming paper. 
	\end{itemize}
\end{remark}

The existence of a {\em $G$-invariant} unit normal can help to distinguish the min-max $G$-hypersurface $\Sigma$ and the fixed points set under certain $\mZ_2$ actions. 
For instance, consider a positive Ricci curvature $3$-ellipsoid $M$ with its major axis (on $x_1$) sufficiently long and the other principal axes bounded by $2$. 
Then the classical min-max theory shall provide the equator $\Gamma=\{x_1=0\}\cap M$ on the major axis as the min-max hypersurface. 
Although $\Gamma$ is also invariant under the $\mZ_2$-reflections $(x_1,x')\mapsto (-x_1,x')$, it can not be the min-max $\mZ_2$-hypersurface since its unit normal is not $\mZ_2$-invariant. 
An interesting question is what exactly is the min-max $\mZ_2$-hypersurface in this case, and how does it relate to $2$-min-max minimal hypersurfaces?

The characterizations of the Morse index and multiplicity for min-max hypersurfaces are crucial in the study of min-max theory. 
For instance, a key part in the proof of the Willmore conjecture by Marques-Neves \cite{marques2014min} is to show the minimal surface in $\mathbb{S}^3$ constructed by the five-parameter families of min-max has Morse index $5$. 
Additionally, by specifying generically the multiplicity (\cite{zhou2020multiplicity}) and index (\cite{marques2016morse}\cite{marques2021morse}) of min-max hypersurfaces, the multi-parameter min-max theory was used to establish the Morse theory for the area functional. 
In the equivariant case, the author \cite{wang2022equivariant} also proved general upper bounds for the $G$-index (Definition \ref{Def: index and G-index}) of equivariant min-max hypersurfaces from multi-parameter families. 
Therefore, in light of Theorem \ref{Thm: main 1 sketch} and \cite{zhou2020multiplicity}, we conjecture that for a generic $G$-invariant Riemannian metric, the minimal $G$-hypersurface constructed from $k$-parameter families of equivariant min-max shall have multiplicity one, $G$-index $k$, and a $G$-invariant unit normal.


Moreover, it has been discovered in numerous studies that the Morse index of a minimal surface is related to its topology. 
For instance, in a closed manifold with positive Ricci curvature, Choi-Schoen \cite{choi1985space} proved the area of a closed minimal surface can be bounded by its genus. 
Therefore, by Ejiri-Micallef \cite[Theorem 4.3]{ejiri2007comparison}, the index of a such minimal surface is also bounded by its genus. 
Additionally, using the {\em conformal volume}, Yau (c.f. \cite[Chapter VIII, Section 4]{schoen1994lectures}) obtained a genus bound for index one minimal surfaces in positive Ricci curvature manifolds. 
More generally, in an orientable $3$-manifold with nonnegative Ricci curvature, it follows from the sharp estimate of Ros \cite[Theorem 15]{ros2006one} that a closed orientable minimal surface of index one must have genus $\leq 3$. 
For a complete two-sided minimal surface in $\R^3$, Chodosh-Maximo showed in \cite{chodosh2016topology} that its genus and the number of ends give a lower bound on its index. 
We refer to \cite{chodosh2018topology}\cite{meeks2019bounds} for more related research.

Hence, as an application, we use the conformal volume initiated by Li-Yau \cite{li1982new} in the orbit space to show a general genus bound of the equivariant min-max surface in a $3$-manifold with positive Ricci curvature, which further indicates an upper bound of the $G$-width and a bound for the singular points of $\Sigma/G$ (see Theorem \ref{Thm: main 2}).

\begin{theorem}\label{Thm: main 2 sketch}
	Let $(M^3, g_{_M})$ be a closed connected oriented Riemannian $3$-manifold with positive Ricci curvature, and $G$ be a finite group acting on $M$ by orientation preserving isometries. 
	Then the equivariant min-max hypersurface $\Sigma$ corresponding to the fundamental class $[M]$ is a connected minimal $G$-hypersurface of multiplicity one with  
	$${\rm genus}(\Sigma )\leq 4 K, \qquad W^G(M) = {\rm Area}(\Sigma) \leq \frac{8\pi K}{c_M},$$
	where $K:= \max_{p\in M}\#G\cdot p\leq \#G$ is the number of points in a principal orbit of $M$, and $\Ric_M\geq c_M>0$. 
	Additionally, the quotient space $\pi(\Sigma)=\Sigma/G$ is an orientable surface with finite cone singular points of order $\{n_i\}_{i=1}^k$ so that $\sum_{i=1}^k (1-\frac{1}{n_i} ) \leq 4 $ and ${\rm genus}(\pi(\Sigma))\leq 3 $. 
	In particular, if $\Sigma\subset M^{prin} $, then ${\rm genus}(\Sigma)\leq 1+2K$. 
\end{theorem}

The conformal method has been employed in many studies for the {\em volume spectrum}, i.e. the multi-parameter version of width. 
For the first width $W(M)$ in the volume spectrum, Glynn-Adey-Liokumovich \cite{glynn2017width} gave an upper bound using the min-conformal volume of the ambient manifold. 
In particular, if $M$ is a closed surface, they showed the first width $W(M)$ can be bounded by the genus and area of $M$. 
Additionally, the conformal upper bounds for the volume spectrum were proved by Wang in \cite{wang2021conformal}.

\medskip
The main idea for Theorem \ref{Thm: main 1 sketch} is as follows. 
For the closed manifold $M$ and the Lie group $G$ in Theorem \ref{Thm: main 1 sketch}, we can take any closed embedded minimal $G$-hypersurface $\Sigma$ in $M$ and use the variation of its first eigenvector field to foliate a $G$-neighborhood of $\Sigma$. 
Using a half-space version of the equivariant min-max theory (Theorem \ref{Thm: min-max theorem for M with boundary}), we argue by contradiction to show this local $G$-equivariant foliation can be extended to a continuous $G$-sweepout of $M$ with mass no more than ${\rm Area}(\Sigma)$ (if $\Sigma$ has a $G$-invariant unit normal) or $2{\rm Area}(\Sigma)$. 
Therefore, it follows from the equivariant min-max theory \cite[Theorem 8]{wang2022min} (see also \cite[Theorem 4.20]{wang2022free}) that the equivariant min-max hypersurface is the minimal $G$-hypersurface of {\em least area} in the sense of (\ref{Eq: least area}). 
Additionally, if the equivariant min-max hypersurface does not admit a $G$-invariant unit normal, it must have even multiplicity by the constructions of equivariant min-max (Theorem \ref{Thm: min-max theorem}). 
However, in this case, we can further use the catenoid estimates of Ketover-Marques-Neves \cite{ketover2020catenoid} to add small $G$-invariant cylinders in the $G$-sweepouts (Proposition \ref{Prop: sweepout construction 2}), which will strictly decrease the mass and give a contradiction. 

The above idea shares the same spirit as in \cite{zhou2015min}. 
However, since the equivariant min-max theory was already established in a continuous version \cite[Theorem 8]{wang2022min}, we do not need to invoke the smooth setting of min-max (c.f. \cite{de2013existence}) as in \cite[Section 2]{zhou2015min}, but give a more self-contained equivariant min-max construction in half spaces (Theorem \ref{Thm: min-max theorem for M with boundary}). 
Meanwhile, instead of using the discretization theorem as in \cite[Theorem 5.8]{zhou2015min}, we can more easily determine that the extension of the $G$-equivariant foliation is a $G$-sweepout.

\medskip
The article is organized as follows. 
In Section \ref{Sec: preliminary}, we collect some notations and definitions of Lie group actions and geometric measure theory. 
In particular, we also introduce the $G$-equivariant sweepouts and $G$-width of $M$ in a continuous version using the isomorphic map between $\pi_1(\Z_n^G(M;\mZ_2))$ and $H_{n+1}(M;\mZ_2)$. 
Then we introduce in Section \ref{Sec: Equivariant min-max} the equivariant min-max theory developed by the author in \cite{wang2022min}\cite{wang2022free} under the Almgren-Pitts setting with some modifications. 
In Section \ref{Sec: sweepouts}, we will generate a continuous $G$-sweepout with good properties from a given minimal $G$-hypersurface. 
Finally, the proof of the main theorem and its applications are given in Section \ref{Sec: proof main}.

\begin{acknowledgements}
	The author would like to thank Professor Gang Tian for his constant encouragement. 
	He also thanks Prof. Xin Zhou, Zhichao Wang, and Linsheng Wang for helpful discussions. 
	Part of this work was completed while the author was visiting the Department of Mathematics at Cornell University, supported by the China post-doctoral grant 2022M722844. Thanks for its hospitality and for providing a good academic environment. 
\end{acknowledgements}

\noindent
{\bf Funding.} Project funded by China Postdoctoral Science Foundation 2022M722844.

\section{Preliminary}\label{Sec: preliminary}

Let $(M^{n+1}, g_{_M})$ be an orientable connected compact Riemannian $(n+1)$-dimensional manifold and $G$ be a compact Lie group acting isometrically on $M$. 
Denote by $\mu$ a bi-invariant Haar measure on $G$ normalized to $\mu(G)=1$. 
For the case that $\partial M \neq \emptyset$, it follows from \cite[Lemma A.1]{wang2022free} that $M$ can be equivariantly and isometrically extended to a closed Riemannian manifold $(N, g_{_N})$ with $G$ acting on $N$ by isometries. 
Therefore, we can assume $M$ is a compact domain of a closed Riemannian $G$-manifold $N$. 

\subsection{Lie group actions}\label{Subsec: 2.1 Lie group actions}
To begin with, we gather some definitions of Lie group actions, most of which are referred from \cite{berndt2016submanifolds}\cite{bredon1972introduction}. 

It follows from \cite{moore1980equivariant} that there is an orthogonal representation $\rho:G\to O(L)$ and an isometric embedding $i:M \hookrightarrow \R^L $ for some $L\in\N$ so that $i$ is equivariant, i.e. $i\circ g = \rho(g)\circ i $. 
For simplicity, we regard $M$ as a subset of $\R^L$ and denote the orthogonal action of $g\in G$ on $x\in \R^L$ as $g\cdot x$. 
We say a subset (hypersurface) $A\subset M$ is a $G$-subset ($G$-hypersurface) if $g\cdot A=A$ for all $g\in G$. 

For any $p\in M$, let $G\cdot p :=\{g\cdot p : g\in G\}$ be the orbit containing $p$ and $G_p:=\{g\in G: g\cdot p = p\}$ be the isotropy group of $p$.
Note $G\cdot p$ is a closed submanifold of $M$ and $G_p$ is a Lie subgroup of $G$. 
We then say $p$ has {\em $(G_p)$ orbit type}, where $(G_p)$ is the conjugacy class of $G_p$ in $G$. 
By \cite[Proposition 2.2.4]{berndt2016submanifolds}, there is a (unique) minimal conjugacy class $(P)$ of isotropy groups so that $M^{prin} = M_{(P)}:=\{ p\in M: (G_p) = (P) \}$ is an open dense $G$-subset of $M$. 
We call any $G\cdot p\subset M^{prin}$ a {\em principal orbit} of $M$ and denote by ${\rm Cohom}(G)$ the co-dimension of a principal orbit, which is known as the {\em cohomogeneity} of the actions of $G$. 

Let $M/G$ be the quotient space, i.e. the {\em orbit space}, and $\pi$ be the projection $\pi:M \to M/G$, $p \mapsto [p]$. 
It is well-known that $M/G$ is a Hausdorff metric space with induced metric $\dist_{M/G}([p],[q]):= \dist_M(G\cdot p, G\cdot q)$. 

Denote by $B_r(p)$, $B_r([p])$, and $\mathbb{B}^k_r(p)$ the geodesic ball in $M$ (or in $N$ if $\partial M \neq\emptyset$), the metric ball in $M/G$, and the Euclidean ball in $\R^k$ respectively. 
Then we use the following notations:
 \begin{itemize}
 	\item $\mathfrak{X}(M), \mathfrak{X}(U)$: the space of smooth vector fields compact supported in $M$ or $U\subset M$; 
 	\item $\mathfrak{X}^G(M), \mathfrak{X}^G(U)$: the space of {\em $G$-vector fields} $X$ in $M$ or $U$, i.e. $g_*X=X, \forall g\in G$; 
	\item $B_\rho^G(p)$: the open geodesic tube with radius $\rho$ around the orbit $G\cdot p$ in $M$ (or in $N$ if $\partial M \neq\emptyset$);
	\item $\an(p,s,t)$: the open tube $B_t^G(p)\setminus \overline{B}^G_s(p)$. 
\end{itemize}
For any closed $G$-hypersurface $\Sigma\subset M$, denote by ${\bf N}\Sigma$ its normal bundle with $G$ acting on it by $g\cdot v:= g_*v$ for all $g\in G, v\in {\bf N}\Sigma$. 
Let $\exp_\Sigma^\perp: {\bf N}\Sigma \to M$ be the normal exponential map of $\Sigma$.
Note $\exp_\Sigma^\perp$ is a $G$-equivariant diffeomorphism in a neighborhood of $\Sigma$.

\subsection{Geometry measure theory}\label{Subsec: 2.2 geometric measure}

We refer \cite{federer2014geometric}\cite{pitts2014existence}\cite{simon1983lectures} for the following definitions:
\begin{itemize}
	\item ${\bf I}_k(M;\mathbb{Z}_2)$: the space of $k$-dimensional mod $2$ flat chains in $\mathbb{R}^L$ with support contained in $M$; 
	\item ${\mathcal Z}_n(M;\mathbb{Z}_2)$: the space of $T\in {\bf I}_n(M;\mathbb{Z}_2)$ with $T= \partial U $ for some $U\in {\bf I}_{n+1}(M;\mathbb{Z}_2)$, i.e. the boundary type mod $2$ $n$-cycles; 
	\item $\mathcal{V}_k(M)$: the weak topological closure of the space of $k$-dimensional rectifiable varifolds in $\mathbb{R}^L$ with support contained in $M$. 
\end{itemize}
Let $\mathcal{F}$ and ${\bf M}$ be the {\it flat (semi-)norm} and the {\it mass} norm in ${\bf I}_k(M;\mathbb{Z}_2)$ (\cite[4.2.26]{federer2014geometric}). 
Define the $\mF$-{\it metric} on $\mathcal{V}_k(M)$ as in \cite[Page 66]{pitts2014existence}. Then $\mF$ induces the weak topology on any mass bounded subset $\{V\in \V_k(M): \|V\|(M)\leq C \}$, where $C>0$ and $\|V\|$ is the Radon measure on $M$ induced by $V$. 

For any $T\in {\bf I}_k(M;\mathbb{Z}_2)$, we denote $|T|$ and $\|T\|$ as the integral varifold and the Radon measure induced by $T$. 
Then we define the ${\bf F}$-{\it metric} on ${\bf I}_k(M;\mathbb{Z}_2)$ by
$$ \mF (S,T):=\mathcal{F}(S-T)+{\bf F}(|S|,|T|), \quad \forall S,T\in {\bf I}_k(M;\mathbb{Z}_2) .$$
It follows from \cite[Page 68]{pitts2014existence} that for any $T, \{T_i\}_{i\in\N} \subset \Z_n(M;\mZ_2) $, 
\begin{equation}\label{Eq: F-metric}
	\lim_{i\to\infty}\mF(T_i , T)= 0 \quad \Leftrightarrow \quad\lim_{i\to\infty}\F(T_i , T)= 0 ~{\rm and }~ \lim_{i\to\infty}\M(T_i) =  \M(T).
\end{equation}
For ${\bf v}={\bf M}$, ${\bf F}$, $\F$, let ${\bf I}_k(M;{\bf v};\mathbb{Z}_2)$, ${\mathcal Z}_n(M;{\bf v};\mathbb{Z}_2)$ be the space with topology induced by ${\bf v}$. 

We say $T\in \mI_k(M;\mZ_2)$ (or $V\in \V_k(M)$) is $G$-invariant if $g_\# T=T$ ($g_\# V=V$) for all $g\in G$. 
Then we have the following subspaces of $G$-invariant elements: 
\begin{itemize}
	\item ${\bf I}_k^G(M;\mathbb{Z}_2) := \{T\in {\bf I}_k(M;\mathbb{Z}_2) : ~g_\#T= T,~\forall g\in G \}$;
	\item ${\mathcal Z}_n^G(M;\mathbb{Z}_2) := \{T\in {\mathcal Z}_n(M;\mathbb{Z}_2) : ~T=\partial U,~{\rm for~some~} U\in {\bf I}_{n+1}^G(M;\mathbb{Z}_2)  \}$; 
	\item $\mathcal{V}_k^G(M) := \{V\in \mathcal{V}_k(M) : ~g_\#V=V,~\forall g\in G\}$.
\end{itemize}
Note $\Z_n^G(M;\mZ_2) \subset \{T\in {\mathcal Z}_n(M;\mathbb{Z}_2) : ~g_\#T= T,~\forall g\in G\}$. 
Since $G$ acts by isometries, ${\bf I}_k^G(M;\mathbb{Z}_2)$, ${\mathcal Z}_n^G(M;\mathbb{Z}_2)$, and $\mathcal{V}_k^G(M)$ are closed subspaces with induced metrics $\M,\F,\mF$. 
Moreover, we have the following isoperimetric lemma (c.f. \cite[Lemma 5]{wang2022min}), which is also valid when $\partial M\neq\emptyset$. 

\begin{lemma}\label{Lem: isoperimetric}
	There are constants $\epsilon_{M}>0$, $C_{M}>1$ such that for any $T_1,T_2\in {\bf I}_n^G(M;\mathbb{Z}_2)$ with $\partial T_1= \partial T_2= 0$, and 
	$$\mathcal{F}(T_1-T_2)<\epsilon_{M}, $$there is a unique $Q\in {\bf I}_{n+1}^G( M;\mathbb{Z}_2)$, called {\em the isoperimetric choice of $T_1,T_2$}, satisfying
	\begin{itemize}
    	\item[(i)] $\partial Q = T_1-T_2$,
    	\item[(ii)] ${\bf M}(Q)\leq C_M\cdot\mathcal{F}(T_1-T_2)$.
  	\end{itemize}
\end{lemma}

\medskip
For any $V\in\V_n(M)$ and $X\in \mathfrak{X}(M)$, the first variation of $V$ along $X$ is given by 
$$ \delta V(X) := \frac{d}{dt}\Big|_{t=0} \|(F_t)_\# V\|(M) = \int_{G_n(M)} {\rm div}_S(X)(p) dV(p,S)  ,$$
where $\{F_t\}$ are the diffeomorphisms generated by $X$, and $G_n(M)$ is the Grassmannian bundle of un-oriented $n$-planes over $M$. 
Suppose $V\in \V_n^G(M)$ is $G$-invariant and $U\subset M$ is an open $G$-subset, then we say 
\begin{itemize}
	\item $V$ is {\em stationary} in $U$, if $\delta V(X) =0  $ for all $X\in \mathfrak{X}(U)$;
	\item $V$ is {\em $G$-stationary} in $U$, if $\delta V(X) =0  $ for all $X\in \mathfrak{X}^G(U)$. 
\end{itemize}
Clearly, a stationary $G$-varifold must be $G$-stationary. 
Meanwhile, for any $X\in \mathfrak{X}(U)$, let 
\begin{equation}\label{Eq: average over G}
	X_G := \int_G (g^{-1})_*X d\mu(g). 
\end{equation} 
A direct compute shows $X_G\in \mathfrak{X}^G(U)$ and $\delta V(X) = \delta V(X_G)$ for any $V\in \V_n^G(M)$ (c.f. \cite[Lemma 2.2]{liu2021existence}). 
Hence, we have: 
\begin{equation}\label{Eq: $G$-stationary}
	\mbox{$V\in \V_n^G(M)$ is stationary in $U$ if and only if it is $G$-stationary in $U$.} 
\end{equation}

\subsection{$G$-Sweepouts and $G$-width}\label{Subsec: 2.3 sweepout and width}
To define the equivariant sweepouts and width, we need to introduce a technical assumption: 

\begin{definition}\label{Def: no mass concentration}
	For any $\F$-continuous map $\Phi: [0,1]\to \Z_n^G(M;\mZ_2)$, define 
	$$ {\bf m}^G(\Phi, r) := \sup \{ \|\Phi(x)\|(B^G_r(p)) : x\in [0,1], p\in M \},$$
	where $B^G_r(p)$ is the geodesic $r$-neighborhood of $G\cdot p$ in $M$ (or in $N$ if $M\subset N$ has non-empty boundary).  
	Then we say $\Phi $ has {\em no concentration of mass on orbits} if $\lim_{r\to 0} {\bf m}^G(\Phi, r) = 0$. 
\end{definition}

By (\ref{Eq: F-metric}) and a continuous argument, we have the following lemma (c.f. \cite[Lemma 8]{wang2022min}), which is quite useful in Section \ref{Sec: Equivariant min-max}. 
\begin{lemma}\label{Lem: no mass concentration}
	If $\Phi: [0,1]\to \Z_n^G(M;\mZ_2)$ is $\mF$-continuous, then $\Phi$ has no concentration of mass on orbits and $\sup_{x\in [0,1]}\M(\Phi(x)) < \infty$. 
\end{lemma}

\medskip
\subsubsection{Closed manifolds}~

In this case, $\partial M = \emptyset$. 
Then for any $\F$-continuous closed curve $\Phi: [0,1]\to \Z_n^G(M;\mZ_2)$, $\Phi(0)=\Phi(1)$, we can take $a_j = \frac{j}{3^k}$, $j=0,1,\dots, 3^k$ with $k\in\N$ large enough so that 
\begin{equation}\label{Eq: divide a curve}
	\F(\Phi(x) - \Phi(y))\leq \epsilon_{M}, \quad \forall x, y\in [a_{j}, a_{j+1}],
\end{equation}
where $\epsilon_M>0$ is given by Lemma \ref{Lem: isoperimetric}. 
By Lemma \ref{Lem: isoperimetric}, there is $Q_j \in \mI_{n+1}^G(M;\mZ_2)$ with $\partial Q_j = \Phi(a_{j+1}) - \Phi(a_j)$ and ${\bf M}(Q_j)\leq C_M \mathcal{F}(\Phi(a_{j+1}) - \Phi(a_j))$, $j=0,1,\dots, 3^k-1$. 
Therefore, $Q:=\sum_{j=0}^{3^k-1} Q_j\in \mI_{n+1}^G(M;\mZ_2)$ satisfies $\partial Q = 0$, which indicates $Q=[[M]]$ or $0$ by the Constancy Theorem \cite[26.27]{simon1983lectures}. 
Hence, we can correspond $\Phi$ to a homology class:
\begin{equation}\label{Eq: isomorphism map}
	F_M(\Phi) := \left[ Q \right] \in H_{n+1}(M^{n+1}; \mZ_2). 
\end{equation}
By the constancy theorem, $F_M(\Phi)$ does not depend on the choice of $k$. 
Moreover, by \cite[Remark 2]{wang2022min} and the arguments in \cite{almgren1962homotopy}, we have $F_M(\Phi)=F_M(\Phi')$ for any closed curve $\Phi'$ that is homotopic to $\Phi$ in $\Z_n^G(M;\F;\mZ_2)$, and $F_M$ induces an isomorphism (\cite[Theorem 9]{wang2022min}): 
$$ F_M : \pi_1(\Z_n^G(M;\mZ_2)) \to H_{n+1}(M;\mZ_2) . $$
In the above isomorphism, we do not need to specify the base point of $\pi_1(\Z_n^G(M;\mZ_2))$. 
This is because $\Z_n^G(M;\mZ_2)$ is the $\F$-path connected component of $\mI_n^G(M;\mZ_2)\cap \Z_n(M;\mZ_2)$ containing $0$ (by Lemma \ref{Lem: isoperimetric} and the contraction argument in \cite[Claim 5.3]{marques2021morse}). 

\begin{definition}[$G$-sweepout]\label{Def: sweepout in closed manifolds}
	A closed $\F$-continuous curve $\Phi: S^1\to \Z_n^G(M;\mZ_2)$ is said to be a {\em $G$-sweepout of $M$} if $F_M(\Phi)=[M]\neq 0$. 
\end{definition}

\begin{remark}\label{Rem: sweepouts are homotopic}
	Since $\Z_n^G(M;\mZ_2)$ is $\F$-path connected, every two $G$-sweepouts are homotopic to each other in $\Z_n^G(M;\F;\mZ_2)$. 
	Hence, the set of $G$-sweepouts of $M$ is exactly the non-trivial homotopy class of closed curves in $\Z_n^G(M;\mZ_2) $. 
\end{remark}

Next, we introduce the min-max $G$-width of $M$ which can be regarded as a critical value for the area functional with respect to {\em all} variations by (\ref{Eq: $G$-stationary}). 

\begin{definition}[$G$-width]\label{Def: width of M}
	Let $\mathcal{P}^G(M)$ be the set of $G$-sweepouts of $M$ with no concentration of mass on orbits. 
	Then we define the {\em $G$-width} of $M$ by 
	$$ W^G(M) := \inf_{\Phi\in \mathcal{P}^G(M)} \sup_{x\in S^1} \M(\Phi(x)) .$$
\end{definition}

\medskip
\subsubsection{Compact manifolds with boundary}~

Now we consider the case that $\partial M \neq \emptyset$, and regard $M$ as a compact domain of a closed Riemannian $G$-manifold $N$. 
Let $F_N$ be given by (\ref{Eq: isomorphism map}), and $\nu_{\partial M}$ be the unit normal of $\partial M$ pointing inward $M$. Then for $\eta>0$ small enough, define 
\begin{equation}\label{Eq: M_eta}
	M_\eta := M \setminus \exp_{\partial M}^\perp([0,\eta)\cdot \nu_{\partial M}) = \{ p\in M: \dist_M(p,\partial M)\geq \eta \}. 
\end{equation}

Let $\Phi_{i}: [0,1]\to \Z_n^G(M;\mZ_2)$, $i=1,2$, be two $\F$-continuous curve so that $\Phi_i(0)= [[\partial M]]$ and $\Phi_i(1) = 0$. 
As the constructions in (\ref{Eq: divide a curve}), we can associate $\Phi_i$ to $Q_i\in \mI_{n+1}^G(M;\mZ_2)$ with $\partial Q_i = [[\partial M]]$. 
Then the Constancy Theorem implies $Q_i=[[M]]$. 
Therefore, the curves product (joint curve) $\Phi_2^{-1}\cdot \Phi_1$  satisfies $F_N(\Phi_2^{-1}\cdot \Phi_1) = 0$, and thus $\Phi_2^{-1}\cdot \Phi_1$ is homotopic to $0$ in $\Z_n^G(N;\F;\mZ_2)$. 
Since $\spt(\Phi_i(x))\subset M$ for all $x\in [0,1]$ and $i=1,2$, we can apply the double cover argument in \cite[Theorem 5.1]{marques2021morse} with Lemma \ref{Lem: isoperimetric} in place of \cite[Corollary 1.14]{almgren1962homotopy}, and see the homotopy map between $\Phi_2^{-1}\cdot \Phi_1$ and $0$ can be taken in $\Z_n^G(M;\F;\mZ_2)$. 
Thus $\Phi_1$ and $\Phi_2$ are homotopic to each other in $\Z_n^G(M;\F;\mZ_2)$. 

Next, we introduce the following definition for $G$-manifold with non-empty boundary, which is generalized from the smooth min-max setting \cite[Definition 2.1, 2.5]{zhou2015min}. 

\begin{definition}\label{Def: sweepout and width in M with boundary}
	Suppose $M$ is a compact Riemannian $G$-manifold with boundary $\partial M\neq \emptyset$. 
	Then we call a $\F$-continuous curve $\Phi: [0,1]\to \Z_n^G(M;\mZ_2)$ a {\em $G$-sweepout of $(M,\partial M)$}, if 
	\begin{itemize}
		\item[(i)] $\Phi(0) = [[\partial M]]$, $\Phi(1) = 0$;
		\item[(ii)] there exist $\epsilon>0$ and a smooth $G$-invariant function $w:[0,\epsilon]\times \partial M\to [0,\infty)$ with $w(0,\cdot )\equiv 0$ and $\frac{\partial}{\partial x}w(0,\cdot) >0$, so that $\Phi(x)$, $x\in [0,\epsilon]$, is induced by the smooth $G$-hypersurface $\exp_{\partial M}^\perp(w(x,\cdot) \nu_{\partial M} )$; 
		\item[(iii)] for any $x_0\in (0,1]$, there exists $\eta >0$ so that $\spt(\Phi(x))\subset\subset M_\eta$ for all $x\in [x_0,1]$. 
	\end{itemize}
	Denote by $\mathcal{P}^G(M,\partial M)$ the set of $G$-sweepouts of $(M,\partial M)$ with no concentration of mass on orbits. 
	Then we define the {\em $G$-width of $(M,\partial M)$} by 
	$$ W^G(M,\partial M) := \inf_{\Phi\in \mathcal{P}^G(M,\partial M)} \sup_{x\in [0,1]} \M(\Phi(x)). $$
\end{definition}

\begin{remark}\label{Rem: sweepouts are homotopic in M with boundary}
	As we mention before, any two $G$-sweepouts $\Phi_1, \Phi_2$ of $(M,\partial M)$ must homotopic to each other in $\Z_n^G(M;\F;\mZ_2)$. 
	Moreover, by reparametrization, the foliation parts of $\Phi_i$, $i=1,2$, are homotopic through $v_t:=(1-t)w_1+t w_2 $, where $t\in [0,1]$ and $w_1,w_2: [0,\frac{1}{3} ]\times \partial M \to [0,\infty)$ are given by Definition \ref{Def: sweepout and width in M with boundary}. 
	The non-foliation parts $\Phi_i\llcorner [\frac{1}{3}, 1 ]$ and $\exp_{\partial M}^\perp(v_t(\frac{1}{3},\cdot )\nu_{\partial M})$ are all in $M_\eta$ for some $\eta>0$, and thus the homotopy between these parts can be taken in $\Z_n^G(M_\eta; \F;\mZ_2)$ (c.f. the constructions in \cite[Theorem 5.1]{marques2021morse} with Lemma \ref{Lem: isoperimetric}). 
	Therefore, we can take a homotopy map $H:[0,1]\times [0,1]\to \Z_n^G(M;\F;\mZ_2)$ so that $H(0,\cdot ) = \Phi_1$, $H(1,\cdot)=\Phi_2$, and for every $t\in [0,1]$, $H(t,\cdot )$ is a $G$-sweepout of $(M,\partial M)$. 
\end{remark}


\section{Equivariant min-max theory}\label{Sec: Equivariant min-max}

In this section, we introduce the equivariant min-max constructions in \cite{wang2022min} (see \cite{wang2022free}\cite{wang2022equivariant} for modified versions). 
Then main purpose is to find an integral $G$-varifold $V\in\V_n^G(M)$ induced by a smooth embedded minimal $G$-hypersurface so that $\|V\|(M)= W^G(M)$ (or $W^G(M,\partial M)$ if $\partial M \neq\emptyset$). 
Since our definitions differ slightly from those in \cite{wang2022min}\cite{wang2022free}, we shall outline the essential steps for the sake of completeness.

Throughout this section, let $\mathcal{P}^G= \mathcal{P}^G(M)$ or $\mathcal{P}^G(M,\partial M)$, $W^G = W^G(M)$ or $W^G(M,\partial M)$ depending on whether $\partial M$ is empty. 
By reparametrization, we always assume the domain of $\Phi \in \mathcal{P}^G$ is $I = [0,1]$, and if $\partial M \neq \emptyset$, then $\Phi\llcorner [0, 1/3]$ are smooth $G$-hypersurfaces as in Definition \ref{Def: sweepout and width in M with boundary}(ii). 

For any sequence $\{\Phi_i\}_{i\in\N}\subset \mathcal{P}^G$, define the {\em width} of $\{\Phi_i \}_{i\in\N}$ by
$$ {\bf L}(\{\Phi_i\}_{i\in\N}) : = \limsup_{i\to\infty} \sup_{x\in I} \M(\Phi_i(x)),$$
Then we say $\{\Phi_i\}_{i\in\N}$ is a {\em min-max} sequence if 
$${\bf L}(\{\Phi_i\}_{i\in\N}) = W^G. $$
The  {\em image set} of $\{\Phi_i\}_{i\in\N}$ is defined by 
$$\mathbf{\Lambda} (\{\Phi_i \}_{i\in\N} ) := \{V \in \mathcal{V}_{n}^G(M): V=\lim_{j\to\infty} |\Phi_{i_{j}}(x_{i_{j}})| \mbox{~for some $i_j\to\infty,x_{i_j}\in I$}\} .$$ 
Moreover, we define the {\em critical set} of $\{\Phi_i \}_{i\in\N}$ by 
$$\mathbf{C}(\{\Phi_i \}_{i\in\N}) := \{V\in\mathbf{\Lambda} (\{\Phi_i \}_{i\in\N} ) : ||V||(M) = {\bf L}(\{\Phi_i\}_{i\in\N}) \}  .$$

\subsection{Discrete min-max settings}\label{Subsec: discrete settings}

To apply the equivariant min-max constructions in \cite{wang2022min}\cite{wang2022free}, we need the following discrete notations. 
Since we only consider curves in $\Z_n^G(M;\mZ_2)$, we will restrict the notations to the $1$-parameter case. 

Denote by $I:=[0,1]$. 
For any $j\in \mathbb{N}$, let $I(1,j)$ be the cube complex on $I$ with $1$-cells and $0$-cells (vertices) given by: 
$$I(1,j)_1:=\{[0,3^{-j}], [3^{-j},2 \cdot 3^{-j}],\dots,[1-3^{-j}, 1]\}, \qquad I(1,j)_0 :=\{ [0], [3^{-j}],\dots, [1] \}.$$
The boundary homeomorphism $\partial$ is defined by $\partial [a,b] = [b] - [a]$. 
Then we denote by $I(2,j) = I(1,j)\otimes I(1,j)$ the cell complex on $I^2=I\times I$. 
For any $\alpha = \alpha_1\otimes\alpha_2 \in I(2,j)$, we say $\alpha$ is a {\em $p$-cell}, if $\dim(\alpha_1 )+\dim(\alpha_2) = p$.
Then the set of $p$-cells of $I(2,j)$ is denoted by $I(2,j)_p$, and the set of $p$-cells in $\alpha\in I(i,j)_q$ is denoted by $\alpha_p$.

Let $J:=[1/3, 1]$. 
Then we denote by $J(1,j)$ the cubical subcomplex containing all the cells of $I(1,j)$ supported in $J$.
Similarly, the set of $p$-cells of $J(1,j)$ is denoted by $J(1,j)_p$ .

Given any two vertices $x, y\in I(m,j)_0$, define the distance ${\bf d}(x,y) := 3^j\sum_{i=1}^m |x_i - y_i|$. 
For any map $\phi : I(1,j)_0 \rightarrow  \mathcal{Z}_n^G(M;\mathbb{Z}_2)$, we define the {\em $\M$-fineness} of $\phi$ by
$${\bf f}_{\M}(\phi) := \sup\left\{{\bf M}(\phi(x)-\phi(y)) :  {\bf d}(x,y)=1,~ x,y\in  I(1,j)_0\right\}.$$
Suppose $S=\{\varphi \}_{i\in\N}$ is a sequence of maps $\varphi_i : I(1,k_i)_0\to\Z_n^G(M;\mZ_2)$ such that $k_i\to\infty$ and ${\bf f}_\M(\varphi_i)\to 0$ as $i\to\infty$. 
Then we use the following notations: 
\begin{itemize}
	\item ${\bf L}(S) : = \limsup_{i\to\infty} \max_{x\in I(1,k_i)_0} \M(\varphi_i(x))$;
	\item $\mathbf{\Lambda} (S ) := \{V \in \mathcal{V}_{n}^G(M): V=\lim_{j\to\infty} |\varphi_{i_{j}}(x_{i_{j}})| \mbox{~for some $i_j\to\infty,x_{i_j}\in I(1,k_i)_0$}\} $;
	\item $\mathbf{C}(S) := \{V\in\mathbf{\Lambda} (S ) : ||V||(M) = {\bf L}(S) \}  .$
\end{itemize}

For any $i,j\in \mathbb{N}$, let ${\bf n}(i,j): I(1,i)_0\rightarrow I(1,j)_0$ be the nearest projection, i.e. $${\bf d}(x,{\bf n}(i,j)(x)) = \inf\{{\bf d}(x,y) : y\in I(m,j)_0\}.$$
Then we define the discrete homotopy: 

\begin{definition}\label{Def: homotpy for maps}
	Given $\phi_i : I(1,k_i)_0 \rightarrow  \Z_{n}^G(M;\mZ_2)$, $i=1,2$, we say $\phi_1$ and $\phi_2$ are {\it $1$-homotopic in $\Z_{n}^G(M;\mZ_2)$ with $\M$-fineness $\delta$} if there exists a map 
	$$\psi: I(1,k)_0\times I(1,k)_0 \rightarrow  \Z_{n}^G(M;\mZ_2)$$
	for some $k\geq\max\{k_1,k_2\}$ such that ${\bf f}_\M (\psi)<\delta$ and $\psi([i-1],x)=\phi_i( {\bf n}(k,k_i)(x))$, for $i\in\{1,2\}$ and $x\in I(1,k)_0$. 
\end{definition}

\begin{definition}\label{Def: homotopy sequence}
	A sequence of mappings $S=\{\phi_i\}_{i\in\N}$,
	$\phi_i: I(1,k_i)_0\rightarrow \Z_{n}^G(M;\mZ_2)$, is a
 	$$\mbox{{\em $(1,{\bf M})$-homotopy sequence of mappings into $\Z_{n}^G(M;\mZ_2)$}},$$ 
	if $\phi_i$ and $\phi_{i+1}$ are $1$-homotopic in $ \Z_{n}^G(M ;\mZ_2)$ with $\M$-fineness $\delta_i$ such that
	\begin{itemize}
		\item[(i)] $\lim_{i\rightarrow\infty} \delta_i=0$;
		\item[(ii)]$\sup\{{\bf M}(\phi_i(x)) :  x\in I(1,k_i)_0, ~i\in \mathbb{N}\}<+\infty.$
	\end{itemize}
\end{definition}

\begin{definition}\label{Def: homotopy for sequence}
	Let $S^j=\{\phi^j_i\}_{i\in\N}$, $j=1,2$, be two $(1,{\bf M})$-homotopy sequences of mappings into $ \Z_{n}^G(M ;\mZ_2) $. 
	Then $S^1$ and $S^2$ are {\em homotopic in $ \Z_{n}^G(M ;\mZ_2) $} if there exists a sequence $\{\delta_i\}_{i\in \mathbb{N}}$ such that
 	\begin{itemize}
		\item[(i)] $\phi^1_i$  is $1$-homotopic to $\phi^2_i$ in $\Z_{n}^G(M;\mZ_2)$ with $\M$-fineness $\delta_i$;
		\item[(ii)] $\lim_{i\rightarrow\infty} \delta_i=0.$
 	\end{itemize}
\end{definition}

By the following discretization theorem from \cite[Theorem 2]{wang2022min}, we can generate a $(1,\M)$-homotopy sequence of mappings into $\Z_{n}^G(M;\mZ_2)$ from any $\Phi\in\mathcal{P}^G$. 

\begin{theorem}[Discretization Theorem]\label{Thm: discretization}
	Let $\Phi: I \rightarrow \mathcal{Z}_n^G(M;\mathbb{Z}_2)$ be a continuous map in the flat topology so that $\sup_{x\in I}\M (\Phi(x))<\infty$ and $\Phi$ has no concentration of mass on orbits. 
	Then there exists a sequence of maps
	$$\phi_i:I(1,j_i)_0 \rightarrow \mathcal{Z}_n^G(M;\mathbb{Z}_2),$$
	with $j_i<j_{i+1}$, and a sequence of positive numbers $\{\delta_i\}_{i\in\mathbb{N}}$ converging to zero such that
	\begin{itemize}
		\item[(i)] $S=\{\phi_i\}_{i\in\mathbb{N}}$ is a $(1,{\bf M})$-homotopy sequence of mappings into $\mathcal{Z}_n^G(M;\mathbb{Z}_2)$ with $\M$-fineness ${\bf f}_\M(\phi_i)<\delta_i$;
		\item[(ii)] there exists some sequence $k_i \to +\infty$ such that for all $x\in I(1, j_i)_0$,
		$$ \M (\phi_i(x)) \leq \sup \{ \M (\Phi(y)) : \al\in I(1, k_i)_1, x, y \in\al\}+\de_i,$$
		which implies $\bL(S)\leq\sup_{x\in I}\M(\Phi(x))$; 
		\item[(iii)] $\sup\{\mathcal F(\phi_i(x)-\Phi(x)) : x\in I(1,j_i)_0\}\leq \delta_i;$
		\item[(iv)] $\Phi(0) =\phi_i([0]) = \psi_i(\cdot, [0])$, $\Phi(1)= \phi_i([1]) = \psi(\cdot ,[1])$, where $\psi_i $ is the discrete homotopy map of $\phi_i$ and $\phi_{i+1}$ given by (i) with $\psi_i([0],{\bf n}(\cdot) )=\phi_i$, $\psi_i([1],{\bf n}(\cdot) )=\phi_{i+1}$.
	\end{itemize}
	Moreover, let $K\subset M$ be a compact $G$-invariant domain with smooth boundary. 
	Then for any $j\in\N$ and $\alpha\in  I(1,j)_1$, if $\spt(\Phi(x))\subset K$ for all $x\in \alpha$, then we can further make $\spt(\phi_i(x))\subset K$ for all $x\in\alpha\cap I(1,j_i)_0$. 
\end{theorem}

\begin{proof}
	The statements in (i)-(iii) follow directly from \cite[Theorem 2]{wang2022min}. 
	Note the proof of \cite[Theorem 2]{wang2022min} is basically the combinatorial approach in \cite[Theorem 13.1]{marques2014min} with Lemma \ref{Lem: isoperimetric} in place of \cite[Corollary 1.14]{almgren1962homotopy} and $\dist(G\cdot p, \cdot)$ in place of $\dist(p, \cdot)$. 
	Meanwhile, since the maps are defined on $1$-dimensional cubical complex, (iv) follows from \cite[Proposition 13.5(ii)]{marques2014min} and the combinatorial constructions of \cite[Theorem 13.1(iv)]{marques2014min}. 
	Moreover, these arguments would also carry over in the case $\partial M\neq \emptyset$, and thus (i)-(iv) are still valid when $M$ has boundary. 
	Finally, if $K$ and $\alpha\in I(1,j)$ are given as in the last statement. 
	Then we can apply the above discretization result to $\Phi\llcorner\alpha$ in $K$ and $\Phi\llcorner(I\setminus \interior (\alpha))$ in $M$ respectively. 
	Note the boundary values are unchanged by (iv). 
	Hence, the discrete maps defined in $\alpha$ and $I\setminus \interior (\alpha)$ can be connected together, which gives the last statement. 
\end{proof}

The following interpolation theorem (c.f. \cite[Theorem 3]{wang2022min}) suggests that a $\M$-continuous map into $\Z_n^G(M;\mZ_2)$ can be generated from a discrete map with small $\M$-fineness. 

\begin{theorem}[Interpolation Theorem]\label{Thm: interpolation}
	For $m=1,2$, there exists a positive constant $C_0=C_0(M,G,m)$ so that if
	$\phi:I(m,k)_0\rightarrow \mathcal{Z}_n^G(M; \mathbb{Z}_2)$
	has ${\bf f}_{\M}(\phi)<\epsilon_{M}$ with $\epsilon_{M}>0$ given in Lemma \ref{Lem: isoperimetric}, then there exists a map
	$$ \Phi: I^m \rightarrow \mathcal{Z}_n^G(M;\mathbb{Z}_2)$$
	continuous in the ${\bf M}$-topology satisfying:
	\begin{itemize}
		\item[(i)] $\Phi(x)=\phi(x)$ for all $x\in I(m,k)_0$;
		\item[(ii)] if $\alpha$ is some $j$-cell in $I(m,k)$, then $\Phi$ restricted to $\alpha$ depends only on the values of $\phi$ assumed on the vertices of  $\alpha$;
		\item[(iii)] $\sup\{{\bf M}(\Phi(x)-\Phi(y)) :  x,y\mbox{ lie in a common cell of } I(m,k)\}\leq C_0{\bf f}_\M(\phi)$;
		\item[(iv)] for any $\alpha\in I(m,k)_j$, if $\phi\llcorner \alpha_0\equiv T\in \Z_n^G(M;\mZ_2)$ is a constant, then $\Phi\llcorner\alpha\equiv T$.
	\end{itemize}
\end{theorem}

We call the map $\Phi$ in Theorem \ref{Thm: interpolation} the {\em Almgren $G$-extension} of $\phi$. 

\begin{proof}
	The statements in (i)-(iii) follow directly from \cite[Theorem 3]{wang2022min}. 
	If $\partial M\neq\emptyset$, then the constructions in \cite[Theorem 4.13]{wang2022free} would carry over with $\Z_n^G(M;\mZ_2)$ and Lemma \ref{Lem: isoperimetric} in place of $\Z_n^G(M,\partial M;\mZ_2)$ and \cite[Lemma 3.10]{wang2022free}. 
	If $\phi\llcorner \alpha_0\equiv T\in \Z_n^G(M;\mZ_2)$ is a constant for some $j$-cell $\alpha$, then for any $1$-cell $\gamma_1=[a,b]\in  \alpha_1$, the isoperimetric choice $Q(\gamma_1 ) $ of $\phi(a)$ and $\phi(b)$ (Lemma \ref{Lem: isoperimetric}) must be $0$. 
	Hence, for any cell $\beta\subset\alpha$, the map $h_\beta$ constructed in \cite[Theorem 4.13]{wang2022free} is $0$, which implies $\Phi\llcorner\alpha \equiv T$ by \cite[4.5]{almgren1962homotopy}. 
\end{proof}

Using the discretization/interpolation theorem \ref{Thm: discretization} and \ref{Thm: interpolation}, we have the following corollary (c.f. \cite[Corollary 1]{wang2022min}):
\begin{corollary}\label{Cor: discrete and continuous}
	Let $\Phi: I\to\Z_n^G(M;\mZ_2)$ be a $\F$-continuous map with no concentration of mass on orbits and $\sup_{x\in I}\M(\Phi(x))<\infty$. 
	Suppose $S=\{\phi_i\}_{i\in\N}$ is given by Theorem \ref{Thm: discretization} applied to $\Phi$, and $\Phi_i$ is the Almgren $G$-extension of $\phi_i$ given by Theorem \ref{Thm: interpolation} for $i$ sufficiently large. 
	Then 
	\begin{itemize}
		\item[(i)] for each $i$ large enough, there is a relative homotopy map $H_i : I^2\to \Z_n^G(M;\F;\mZ_2)$ with $H_i(0,\cdot)=\Phi$, $H_i(1,\cdot)=\Phi_i$, $H_i(\cdot , 0)\equiv \Phi(0)=\Phi_i(0)$, and $H_i(\cdot , 1)\equiv \Phi(1)=\Phi_i(1)$; 
		\item[(ii)] ${\bf L}(\{\Phi_i\}_{i\in\N}) = {\bf L}(S) \leq \sup_{x\in I}\M(\Phi(x))$.
	\end{itemize}
\end{corollary}
	
\begin{proof}
	Using Theorem \ref{Thm: interpolation} and the arguments of Almgren \cite{almgren1962homotopy}, we see \cite[Theorem 8.2]{almgren1962homotopy} is valid in our $G$-invariant settings (even if $\partial M$ may be non-empty). 
	Hence, the proof of \cite[Corollary 3.9]{marques2017existence} would carry over with Theorem \ref{Thm: discretization} and \ref{Thm: interpolation} in place of \cite[Theorem 3.6, 3.6]{marques2017existence}. 
	Thus, $\Phi_i$ is homotopic to $\Phi$ in $\Z_n^G(M;\F;\mZ_2)$ for $i$-large, and (ii) is valid. 
	Additionally, by (iv) in Theorem \ref{Thm: discretization} and \ref{Thm: interpolation}, we have $\Phi(0)=\phi_i([0])=\Phi_i(0)$ and $\Phi(1)=\phi_i([1])=\Phi_i(1)$ for all $i$-large. 
	Therefore, combining (iv) in Theorem \ref{Thm: discretization} and \ref{Thm: interpolation} with the homotopy constructions in \cite[Proposition 3.3, 3.8]{marques2017existence}, one easily verifies that the homotopy map $H_i$ of $\Phi$ and $\Phi_i$ is relative to the boundary values. 
\end{proof}

Let $\{\Phi_i\}_{i\in\N}\subset \mathcal{P}^G$ be any min-max sequence. 
If $\partial M=\emptyset$, then we can apply Corollary \ref{Cor: discrete and continuous} to each $\Phi_i$ and obtain a sequence of $\M$-continuous curves $\{\Phi_j^i\}_{j\in\N}$ relative homotopic to $\Phi_i$ in $\Z_n^G(M;\F;\mZ_2)$ and ${\bf L}(\{\Phi_j^i\}_{j\in\N}) \leq \sup_{x\in I}\M(\Phi_i(x))$. 
Choose $j(i)$ sufficiently large so that $\sup_{x\in I}\M(\Phi_{j(i)}^i(x)) \leq \sup_{x\in I}\M(\Phi_i(x)) + \frac{1}{i}$. 
Hence, we have $\{\Phi_{j(i)}^i\}_{i\in\N}\subset \mathcal{P}^G(M)$ is a min-max sequence continuous in the $\M$-topology and so in the $\mF$-topology. 

For the case $\partial M\neq\emptyset$, we can apply the above arguments to each $\Phi_i\llcorner J$ in a $G$-submanifold $M_{\eta_i}$ given by Definition \ref{Def: sweepout and width in M with boundary}(iii) with $x_0=\frac{1}{3}$, and get $\Phi_{j(i)}^i: J \to \Z_n^G(M_{\eta_i};\M;\mZ_2 )$ satisfying 
\begin{itemize}
	\item $\Phi_{j(i)}^i$ is relative homotopic to $\Phi_i\llcorner J$ in $\Z_n^G(M_{\eta_i};\F;\mZ_2 )$,
	\item $\sup_{x\in J}\M(\Phi_{j(i)}^i(x)) \leq \sup_{x\in J}\M(\Phi_i(x)) + \frac{1}{i}$.
\end{itemize}
Since the homotopy map of $\Phi_{j(i)}^i$ and $\Phi_i\llcorner J$ is relative to the boundary values, we can define $ \Phi_{j(i)}^i \llcorner [0, \frac{1}{3} ] = \Phi_i\llcorner [0, \frac{1}{3} ] $ and see $\{\Phi_{j(i)}^i\}_{i\in\N}\subset \mathcal{P}^G(M,\partial M)$ is a sequence of $\mF$-continuous min-max sequence. 

Therefore, the above arguments give the following corollary, which implies we only need to consider the $\mF$-continuous $G$-sweepouts. 
\begin{corollary}\label{Cor: F-continuous width}
	The $G$-width defined in Definition \ref{Def: width of M} and \ref{Def: sweepout and width in M with boundary} satisfies 
	$$W^G = \inf \{ \sup_{x\in I}\M(\Phi(x)) : \Phi\in \mathcal{P}^G \mbox{ is $\mF$-continuous} \}.$$
\end{corollary}

\subsection{Min-max theorems}\label{Subsec: min-max constructions}

We now use the min-max method to construct a minimal $G$-hypersurface (with multiplicity) so that the width $W^G$ is realized by its area. 


\medskip
\subsubsection{Closed manifolds}~

For the case that $M$ is closed, it follows from Remark \ref{Rem: sweepouts are homotopic} and Corollary \ref{Cor: F-continuous width} that ${\bm \Pi} := \{ \Phi\in \mathcal{P}^G(M) \mbox{ is $\mF$-continuous} \}$ is a {\em continuous $G$-homotopy class} in the sense of \cite[Definition 5]{wang2022min}, and $W^G(M) = {\bf L}({\bm \Pi})$ in the sense of \cite[Definition 6]{wang2022min}.  
Hence, we have the following min-max theorem by \cite[Theorem 8]{wang2022min}. (Note the assumptions on $M\setminus M^{prin}$ in \cite[Theorem 8]{wang2022min} can be removed by the modifications in \cite{wang2022free}, and the dimension assumption is modified in \cite[Theorem 5.1]{wang2022equivariant}.) 
\begin{theorem}\label{Thm: min-max theorem}
	Suppose $M$ is closed, i.e. $\partial M = \emptyset$, and $3\leq {\rm codim(G\cdot p)}\leq 7$ for all $p\in M$. 
	Then there exists an integral $G$-varifold $V\in\V_n^G(M)$ so that 
	$$\|V\|(M) = W^G(M), \quad {\rm and} \quad V=\sum_{i=1}^m n_i|\Sigma_i|,$$ 
	where $m,n_i\in\N$, $\{\Sigma_i\}_{i=1}^m$ are disjoint $G$-connected (Definition \ref{Def: G-connected}) smooth embedded closed minimal $G$-hypersurfaces. 
	Moreover, if $\Sigma_i$ does not admit a $G$-invariant unit normal vector field, then $n_i$ is an even number. 
\end{theorem}
\begin{proof}
	We only need to show the last statement since the existence and regularity of $V$ are given by \cite[Theorem 8]{wang2022min} (see also \cite[Theorem 5.1]{wang2022equivariant}). 
	Note the min-max varifold $V$ is $(G,\mZ_2)$-almost minimizing in annuli of {\em boundary type} in the sense of \cite[Definition 10,11]{wang2022min}. 
	Hence, for each $\Sigma_i$, we can take a small $G$-tube $B^G_{2r}(p)$ with center $G\cdot p\subset \Sigma_i$ and $r\in (0, \frac{\inj(G\cdot p)}{2} )$ so that 
	\begin{itemize}
		\item $V$ is $(G,\mZ_2)$-almost minimizing of boundary type in $B^G_{2r}(p)$;
		\item $B^G_{t}(p)$ has mean convex boundary for all $t\in (0, 2r)$;
		\item $B^G_{2r}(p)\cap \spt(\|V\|)\subset \Sigma_i$, and $\partial B^G_r(p)$ is transversal to $\Sigma_i$. 
	\end{itemize}
	Then by the constructions (\cite[Proposition 2, 3]{wang2022min}) and the consistency (\cite[Proposition 4.19]{wang2022equivariant}) of $G$-replacements, there exists a sequence $\{T_j\}_{j\in\N}\subset \Z_n^G(M;\mZ_2)$ so that 
	\begin{itemize}
		\item[(1)] $T_j=\partial Q_j$ is locally mass minimizing in $B^G_r(p)$ with $Q_j\in\mI_{n+1}^G(M;\mZ_2)$;
		\item[(2)] $|T_i|\to V$ in the sense of varifolds. 
	\end{itemize}
	By compactness, let $T_j\to T=\partial Q$ in the flat topology with $Q\in \mI_{n+1}^G(M;\mZ_2)$. 
	Thus, we have $\spt(T)\subset \spt(V) = \cup_{i=1}^m\Sigma_i$, which implies $T=\sum_{i=1}^m n_i'[[\Sigma_i]]$ for some $n_i'\in\mZ_2$ by the Constancy Theorem. 
	As a boundary of $Q\in  \mI_{n+1}^G(M;\mZ_2)$, we must have $n_i'=0$ if $\Sigma_i$ does not admit a $G$-invariant unit normal. 
	Now we can use the slicing theory (\cite[28.5]{simon1983lectures}) to take $s\in (r/2, r)$ so that $\M(\partial (T_j\llcorner B_s^G(p)))$ are uniformly bounded, and thus $T_j\llcorner B_s^G(p)$ converges up to a subsequence. 
	Finally, because of (1), 
	\cite[Theorem 1.1]{white2009currents} suggests $n_i\equiv n_i'$ mod $2$, and thus the multiplicity $n_i$ must be even for $\Sigma_i$ without a $G$-invariant unit normal. 
\end{proof}

\medskip
\subsubsection{Compact manifolds with boundary}~

Now we consider the case that $\partial M\neq\emptyset$. 
In this case, we make the assumption that 
\begin{equation}\label{Eq: assumption on boundary}
	H_{\partial M}>0, \quad {\rm and }\quad W^G(M,\partial M) > {\rm Area}(\partial M),
\end{equation}
where $H_{\partial M}$ is the mean curvature of $\partial M$ with respect to the inward unit normal $\nu_{\partial M}$. 
By Corollary \ref{Cor: F-continuous width}, we can take a min-max sequence $\{\Phi_i^*\}_{i\in\N}\subset \mathcal{P}^G(M,\partial M)$ that are continuous in the $\mF$-topology. 
The strategy is to use the following proposition to deform $\{\Phi_i^*\}_{i\in\N}$ into a new $\mF$-continuous min-max sequence so that every $V\in {\bf C}(\{\Phi_i^*\}_{i\in\N})$ is supported in a $G$-invariant subdomain $M_a\subset\subset M$. 
With this benefit, the min-max constructions can be restricted in the interior of $M$ to build a closed minimal $G$-hypersurface realizing the width $W^G(M,\partial M)$. 
This deformation approach is based on the idea of \cite[Lemma 2.2]{marques2012rigidity} and we list the details here for the sake of completeness. 
\begin{proposition}\label{Prop: deform sweepout for M with boundary}
	Suppose $\partial M \neq \emptyset$ satisfies (\ref{Eq: assumption on boundary}). 
	Then there exist a constant $a>0$ and a min-max sequence $\{\Phi_i^*\}_{i\in\N}\subset \mathcal{P}^G(M,\partial M)$ continuous in the $\mF$-topology so that 
	\begin{itemize}
		\item if $\M(\Phi_i^*(x)) \geq W^G(M,\partial M) - \delta $ with $\delta = \frac{1}{4}(W^G(M,\partial M) - {\rm Area}(\partial M))$, then $\spt(\Phi_i^*(x))\subset \subset M_a := \{p\in M:\dist_M(p,\partial M)\geq a\}$. 
	\end{itemize} 
\end{proposition}
\begin{proof}
	Let $a>0$ be small enough so that $d := \dist_M(\partial M , \cdot)$ is a $G$-invariant smooth function in a $4a$-neighborhood of $\partial M$. 
	By (\ref{Eq: assumption on boundary}), we can set $a>0$ even smaller so that for any $r\in [0, 3a]$, $\partial M_r = d^{-1}(r)$ has positive mean curvature $H_{r}$ with respect to the inner unit normal $\nabla d$. 
	Denote by $A_r$ the second fundamental form of $\partial M_r$, and $c=\sup_{r\in [0,3a],p\in \partial M_r}|A_r|(p)$. 
	Then we take the function $\phi \geq 0$ as in \cite[Lemma 2.2]{marques2012rigidity} so that 
	\begin{equation*}
		\phi ' + c\phi \leq 0, \qquad \phi(r) > 0 \mbox{ for $r < 2a$},\qquad \phi(r) = 0 \mbox{ for $r\geq 2a$}.
	\end{equation*}
	For any $p\in \interior (M)\setminus M_{3a}$ and $n$-subspace $S\subset T_pM$, let $\{e_i\}_{i=1}^n$ be an orthonormal basis of $S$, and $P: T_pM\to T_p\partial M_{d(p)}$ be the projection. 
	Since $\dim(S\cap T_p\partial M_{d(p)}) \geq n-1$, we can assume $\{e_i\}_{i=1}^{n-1}\cup \{e^*\}$ gives an orthonormal basis of $ T_p\partial M_{d(p)}$, where $e^*$ satisfies $\langle e^*, P(e_{n}) \rangle = |P(e_{n})|$. 
	Noting $\nabla d \perp T_p\partial M_{d(p)}$ and $\nabla_{\nabla d}\nabla d = 0$, we have
	\begin{eqnarray}\label{Eq: divergence}
		{\rm div}_S(\phi \nabla d ) &=& \phi'(d(p)) \cdot \langle e_{n}, \nabla d\rangle^2 + \phi(d(p)) \cdot \sum_{i=1}^n \langle \nabla_{e_i}\nabla d, e_i\rangle \nonumber
		\\ &=& \phi' \langle e_{n}, \nabla d\rangle^2 - \phi \sum_{i=1}^n A_{d(p)}(P(e_i), P(e_i))
		\\ &=& (\phi' + \phi A_{d(p)}(e^*, e^*) )\langle e_{n}, \nabla d\rangle^2 - \phi H_{d(p)} \nonumber
		\\ &\leq & (\phi' + c\phi )\langle e_{n}, \nabla d\rangle^2 - \phi H_{d(p)} \leq 0. \nonumber
	\end{eqnarray}
	By Corollary \ref{Cor: F-continuous width}, we can take any min-max sequence $\{\Phi_i\}_{i\in\N}\subset \mathcal{P}^G(M,\partial M)$ continuous in the $\mF$-metric. 
	Then for each $\Phi_i$, there exist $\epsilon_i >0$ and $\eta_i \in (0, \frac{a}{8} )$ so that 
	\begin{itemize}
		\item[(1)] $\Phi_i\llcorner [0, 4\epsilon_i ]$ are smooth $G$-hypersurfaces with $\M(\Phi_i(x)) \leq {\rm Area}(\partial M) + \delta$, $\forall x\in [0, 4\epsilon_i]$;
		\item[(2)] $\spt(\Phi_i(x))\subset\subset M_{2\eta_i}$ for all $x\in [\epsilon_i , 1 ]$. 
	\end{itemize}
	Let $\kappa_i$ be a cut-off function so that $\kappa_i(r)=0$ for $r\leq \eta_i$, and $\kappa_i(r)=1$ for $r\geq 2\eta_i$. 
	Then the $G$-vector field $X_i:=\kappa_i(d) \phi(d) \nabla d$ generates $G$-equivariant diffeomorphisms $\{F_t^i\}$. 
	By (2) and (\ref{Eq: divergence}), for any $x\in [\epsilon_i , 1 ]$ and $t_0 \geq 0$, we have
	\begin{eqnarray*}
		\frac{d}{dt}\Big|_{t=t_0} \M((F_t^i)_\#\Phi_i(x) ) &=& \frac{d}{dt}\Big|_{t=0}\| (F_t^i)_\# (F_{t_0}^i)_\#\Phi_i(x) \|(M) 
		\\ &=& \int {\rm div}_S(X_i)dV_{t_0,x} = \int {\rm div}_S(\phi \nabla d )dV_{t_0,x} \leq 0, 
	\end{eqnarray*}
	where $V_{t_0,x} := |(F_{t_0}^i)_\#\Phi_i(x)|\in \V_n^G(M_{2\eta_i})$. 
	Therefore, 
	\begin{equation}\label{Eq: decrease mass deformation}
		\M((F_{t}^i)_\#\Phi_i(x)) \leq \M(\Phi_i(x)), \quad \forall x\in [\epsilon_i , 1 ], ~t\geq 0. 
	\end{equation}
	Since $M_{2\eta_i}\setminus M_{2a} \subset \spt(X_i)\subset M_{\eta_i}\setminus \interior(M_{2a})$, we see $\lim_{t\to\infty} F_t^i(p)\in \partial M_{2a}$ for any $p\in M_{2\eta_i}\setminus M_{2a}$, and thus $F_{T_i}^i(M_{2\eta_i})\subset M_a$ for some $T_i>0$. 
	Choose a smooth function $h_i: [0,1]\to [0,T_i]$ with $h_i\llcorner [0, \epsilon_i ] =0$, $h_i\llcorner [2\epsilon_i, 1 ] = T_i$. 
	Then $\Phi_i^* (x) := (F^i_{h_i(x)})_\#\Phi_i(x)$ satisfies:
	\begin{itemize}
		\item[(a)] $\Phi_i^*(x) = \Phi_i(x)$ for $x\in [0,\epsilon_i ]$ (since $h_i=0$); 
		\item[(b)] $\M(\Phi_i^*(x)) \leq \M(\Phi_i(x))$ for all $x\in [\epsilon_i ,1]$ (by (\ref{Eq: decrease mass deformation}));
		\item[(c)] $\spt(\Phi_i^*(x)) \subset \subset M_a$ for all $x\in [2\epsilon_i , 1]$ (by (2) and the definitions of $T_i, h_i$). 
	\end{itemize}
	Clearly, $\{\Phi_i^*\}_{i\in\N}\subset \mathcal{P}^G(M,\partial M)$ is also an $\mF$-continuous min-max sequence. 
	Additionally, if $\M(\Phi_i^*(x))\geq W^G(M,\partial M) - \delta \geq {\rm Area}(\partial M) + \delta$, then $x\in (4\epsilon_i, 1]$ by (1)(a)(b), and thus $\spt(\Phi_i^*(x)) \subset \subset M_a$ by (c). 
\end{proof}

Next, we use the pull-tight arguments to make every $V\in {\bf C}(\{\Phi_i^*\}_{i\in\N})$ stationary in $M$. 
By Proposition \ref{Prop: deform sweepout for M with boundary}, the pull-tight procedure can be restricted in a $G$-subset $\interior(M_a)$. 

\begin{proposition}\label{Prop: tightening}
	Suppose $\partial M \neq \emptyset$ satisfies (\ref{Eq: assumption on boundary}) and $\delta := \frac{1}{4}(W^G(M,\partial M) - {\rm Area}(\partial M))$. 
	Let $a>0$ and $\{\Phi_i^*\}_{i\in\N}\subset \mathcal{P}^G(M,\partial M)$ be 
	given by Proposition \ref{Prop: deform sweepout for M with boundary}. 
	Then there is an $\mF$-continuous min-max sequence $\{\Phi_i\}_{i\in\N}\subset \mathcal{P}^G(M,\partial M)$ with
	\begin{itemize}
		\item[(i)] ${\bf C}(\{\Phi_i\}_{i\in\N}) \subset {\bf C}(\{\Phi_i^*\}_{i\in\N}) \cap \V_n^G( M_a )$;
		\item[(ii)] every $G$-varifold $V\in {\bf C}(\{\Phi_i\}_{i\in\N})$ is stationary in $M$;
		\item[(iii)] if $\M(\Phi_i(x)) \geq W^G(M,\partial M) - \delta $, then $\spt(\Phi_i(x))\subset \subset M_{\frac{a}{2}}$. 
	\end{itemize}
\end{proposition}

\begin{proof}
	Let $C:=\sup_{i\in \N}\sup_{x\in I}\M(\Phi_i^*(x)) < \infty$ and $\mathring{M}_{\frac{a}{2}} := \interior(M_{\frac{a}{2}})$ be a $G$-invariant open set of $M$. 
	Define then $A :=\{V\in\mathcal{V}^G_n(M) : \|V\|(M)\leq C \}$ and 
	$$ A_0 := \{V\in A : V {\rm ~is ~stationary~in~} \mathring{M}_{\frac{a}{2}} \}.$$
	Since $G$ acts by isometries, $A$ and $A_0$ are compact subset of $\V_n^G(M)$. 
	Additionally, for any $V\in A$, it follows from (\ref{Eq: average over G}) that $V\in A_0$ if and only if $\delta V(X) =0$, $\forall X\in \mathfrak{X}^G(\mathring{M}_{\frac{a}{2}})$. 
	Hence, we can follow \cite[Page 765]{marques2014min} (or \cite[Page 153]{pitts2014existence}) with $\mathfrak{X}^G(\mathring{M}_{\frac{a}{2}})$ in place of $\mathfrak{X}(M)$ to define a continuous map $X: A \to \mathfrak{X}^G(\mathring{M}_{\frac{a}{2}})$ and a continuous function $\eta:A \to[0,1]$ satisfying 
	\begin{itemize}
		\item $X(V)=0$ and $\eta(V)=0$, if $V\in A_0$;
		\item $\delta V(X(V))<0$ and $\eta(V)>0$, if $V\in A\setminus A_0$;
		\item $\|(f^{X(V)}_{t})_\#V\|(M) < \|(f^{X(V)}_{s})_\#V\|(M) $ for all $V\in A$ and $0\leq s<t\leq \eta(V)$, 
	\end{itemize}
	where $\{f^{X(V)}_t\}$ are the equivariant diffeomorphisms generated by $X(V)$. 
	Define then 
	\begin{eqnarray}\label{Eq-pulltightmap}
		H:~I\times \{T\in \Z_n^G(M; \mF; \mZ_2) :  {\bf M}(T)\leq C \} 
		&\rightarrow & \{T\in \Z_n^G(M; \mF; \mZ_2) :  {\bf M}(T)\leq C \} , \nonumber
		\\
		H(t,T) &:=& \left(f^{X(|T|)}_{\eta(|T|)t} \right)_\# T. \nonumber
	\end{eqnarray}
	One easily verifies $H(0,T)=T$ for all $T \in \Z_n^G(M; \mZ_2)$ with ${\bf M}(T)\leq C$, and
	\begin{itemize}
		\item if $|T|$ is stationary in $\mathring{M}_{\frac{a}{2}}$, then $H(t,T)=T$ for all $t\in[0,1]$;
		\item if $|T|$ is not stationary in $\mathring{M}_{\frac{a}{2}}$, then ${\bf M}(H(1,T)) < \M(T)$.
	\end{itemize}
	Define $\Phi_i :=H(1,\Phi_i^*)$. 
	Note $X(V)$ is compact supported in $\mathring{M}_{\frac{a}{2}}$ and $f^{X(V)}_{t}\llcorner (M\setminus \mathring{M}_{\frac{a}{2}}) = id$. Hence, $\Phi_i$ is also a $G$-sweepout of $(M,\partial M)$. 
	Additionally, by the above constructions, one easily verifies that $\{\Phi_i\}_{i\in\N}\subset \mathcal{P}^G(M;\partial M)$ is a min-max sequence continuous in the $\mF$-topology, and ${\bf C}(\{\Phi_i\}_{i\in\N}) \subset {\bf C}(\{\Phi_i^*\}_{i\in\N}) \cap A_0$. 
	Moreover, it follows from Proposition \ref{Prop: deform sweepout for M with boundary} that ${\bf C}(\{\Phi_i\}_{i\in\N})\subset \V^G_n(M_a) \cap A_0$, which implies every $V\in {\bf C}(\{\Phi_i\}_{i\in\N})$ is stationary in $M$. 
	Finally, since the deformations $f^{X(V)}_{t}$ are restricted in $\mathring{M}_{\frac{a}{2}}$, the last bullet follows directly from Proposition \ref{Prop: deform sweepout for M with boundary} and the above constructions. 
\end{proof}

Finally, we can now show the equivariant min-max theorem for compact manifold $M$ with boundary $\partial M$ satisfying (\ref{Eq: assumption on boundary}). 
The proof is generally the approach in \cite[Theorem 3.8]{marques2016morse}, and we list some necessary modifications. 
\begin{theorem}\label{Thm: min-max theorem for M with boundary}
	Suppose $\partial M \neq \emptyset$ satisfies (\ref{Eq: assumption on boundary}), and $3\leq {\rm codim}(G\cdot p)\leq 7$ for all $p\in M$. 
	Then there exists an integral $G$-varifold $V\in\V_n^G(M)$ so that $\|V\|(M) = W^G(M,\partial M)$ and $V=\sum_{i=1}^m n_i|\Sigma_i|$, where $m,n_i\in\N$, $\{\Sigma_i\}_{i=1}^m$ are disjoint smooth embedded closed minimal $G$-hypersurfaces in the interior of $M$.  
\end{theorem}

\begin{proof}
	Let $a>0$ and $\{\Phi_i\}_{i\in\N}\subset \mathcal{P}^G(M,\partial M)$ be 
	given by Proposition \ref{Prop: tightening} so that every $V\in {\bf C}(\{\Phi_i\}_{i\in\N})$ is stationary in $M$ and compactly supported in $\interior ( M_{a_0})$ for $a_0 = \frac{a}{2}$. 
	Let $\delta = \frac{1}{4}(W^G(M,\partial M) - {\rm Area}(\partial M))>0$. 
	Then by reparametrization, we assume $\Phi_i\llcorner [0, \frac{1}{3} ]$ foliates a neighborhood of $\partial M$ so that 
	\begin{equation}\label{Eq: small area near boundary}
		\M(\Phi_i(x)) \leq {\rm Area}(\partial M) + \delta = W^G(M,\partial M) - 3\delta, \quad \forall x\in [0, 1/3 ].
	\end{equation}
	
	Denote by 
	$$\Phi_i' := \Phi_i \llcorner J.$$
	By Definition \ref{Def: sweepout and width in M with boundary}, there exists $\eta_i \in (0, {a_0})$ so that $\spt(\Phi_i'(x))\subset\subset M_{\eta_i}$ for all $x\in J$. 
	Additionally, since the map $x\mapsto \M(\Phi_i'(x))$ is continuous (by (\ref{Eq: F-metric})), we can take $k_i\in\N$ large enough so that $|\M(\Phi_i'(x)) -\M(\Phi_i'(y))| \leq \delta /4$ for all $x,y$ in a common $1$-cell of $J(1,k_i)$. 
	Denote by $U_i$ the union of $1$-cells $\alpha\in  J(1,k_i)_1$ so that $\M(\Phi_i'(x))\leq W^G(M,\partial M) - 3\delta/4$ for all $x\in \alpha$, and $V_i :=J\setminus U_i$. 
	Therefore, by Proposition \ref{Prop: tightening}(iii), we have 
	$$\M(\Phi_i'(x))\geq W^G(M,\partial M) - \delta, \quad {\rm and }\quad \spt(\Phi_i'(x))\subset M_{a_0}, \quad \forall x\in V_i.$$
		
	By Lemma \ref{Lem: no mass concentration}, we can apply Theorem \ref{Thm: discretization} to each $\Phi_i'$ in the $G$-submanifold $M_{\eta_i}$ and obtain a sequence of maps $\phi_j^i: J(1, k_j^i)_0 \to \Z_n^G(M_{\eta_i}; \mZ_2)$ with $k_j^i<k_{j+1}^i$, $j\in\N$. 
	The last statement in Theorem \ref{Thm: discretization} indicates $\{\phi_j^i\}_{j\in\N}$ can be chosen to satisfy $\spt(\phi_j^i(x)) \subset M_{a_0}$ for all $x\in V_i\cap J(1, k_j^i)_0$. 
	Moreover, we claim the following result:
	\begin{claim}\label{Claim: restrict discrete map in Ma}
		For $j$ large enouth, if $\M(\phi_j^i(x))\geq W^G(M,\partial M) - \delta/2$ then $\spt(\phi^i_j(x))\subset M_{a_0}$.
	\end{claim}
	\begin{proof}[Proof of Claim 1]
		By the continuity of $x\mapsto \M(\Phi_i'(x))$ and Theorem \ref{Thm: discretization}(ii), if $\M(\phi_j^i(x))\geq W^G(M,\partial M)-\delta/2$, then we have $\M(\Phi_i'(x)) > W^G(M,\partial M)- 3\delta/4$ for $j$ large enough. 
		Thus, such vertex $x $ must be in $V_i$, so $\spt(\phi^i_j(x))\subset M_{a_0}$. 
	\end{proof}
	
	Additionally, we also have the following equality due to the lower semi-continuity of mass, the continuity of $x\mapsto \M(\Phi_i'(x))$, and Theorem \ref{Thm: discretization}(ii)(iii):
	\begin{equation}\label{Eq: F estimates}
		\lim_{j\to\infty} \sup \{ \mF(\phi^i_j(x), \Phi_i'(x)) : x\in J(1,k^i_j)_0  \} = 0. 
	\end{equation}
	Let $\Phi^i_j: J\to \Z_n^G(M_{\eta_i};\M;\mZ_2)$ be the Almgren $G$-extension of $\phi^i_j$ given by Theorem \ref{Thm: interpolation} for $j$-large. 
	By Corollary \ref{Cor: discrete and continuous}, $\Phi^i_j$ and $\Phi_i'$ are {\em relative} homotopic in $\Z_n^G(M_{\eta_i};\F;\mZ_2)$. 
	Therefore, 
	$$ \widetilde{\Phi}_i^j(x) := \left\{\begin{array}{ll}{ \Phi_i(x) ,} & {x\in [0,1/3],} \\ {\Phi_i^j(x), }&{x\in J=[1/3, 1],}\end{array}\right.$$
	is a well-defined $\mF$-continuous $G$-sweepout of $(M,\partial M)$ for each $i\in\N$ and $j$-large, and thus
	\begin{eqnarray}\label{Eq: width estimate}
		W^G(M,\partial M) &\leq & {\bf L} (\{\widetilde{\Phi}^{i}_{j} \}_{j\in \N} )= {\bf L} (\{\Phi^{i}_{j} \}_{j\in \N} ) \nonumber
		\\ &=& {\bf L} (\{\phi^i_j\}_{j\in\N}) \leq \sup\{ \M (\Phi_{i}(x) ): x\in I \}\rightarrow W^G(M,\partial M) ,
	\end{eqnarray} 
	by (\ref{Eq: small area near boundary}) and Corollary \ref{Cor: discrete and continuous}. 
	
	Now, we take a subsequence $j(i)\to \infty$ and define $\widetilde{\Phi}_i=\widetilde{\Phi}^i_{j(i)}$, $S=\{\varphi_i \}_{i\in\N}$, $\varphi_i := \phi^i_{j(i)}$, so that $\mf_\M(\varphi_i) \to 0$ and 
	\begin{itemize}
		\item[(1)] $C_i\mf_\M(\varphi_i) \to 0$ as $i\to\infty$, where $C_i=C_0(M_{\eta_i},G,1)$ is given by Theorem \ref{Thm: interpolation};
		\item[(2)] if $\M(\varphi_i(x))\geq W^G(M,\partial M) - \delta/2$ then $\spt(\varphi_i(x))\subset M_{a_0}$ (by Claim \ref{Claim: restrict discrete map in Ma});
		\item[(3)] $W^G(M,\partial M) = {\bf L}(\{\varphi_i\}_{i\in\N}) $ (by (\ref{Eq: width estimate}));
		\item[(4)] $\lim_{i\to\infty} \sup \{ \mF(\varphi_i(x), \Phi_i(x)) : x\in J(1, k^i_{j(i)})_0  \} = 0$ (by (\ref{Eq: F estimates}));
		\item[(5)] $\lim_{i\to\infty} \sup \{ \mF(\Phi_i(x), \Phi_i(y)) : x,y\in \alpha, \alpha \in  I(1, k^i_{j(i)})  \} = 0$ (by the $\mF$-continuity). 
	\end{itemize}
	Combining (3)(4)(5) with (\ref{Eq: small area near boundary}), we have ${\bf C}(S)={\bf C}(\{\Phi_i\}_{i\in\N})\subset \V_n^G(M_{2a_0})$ and every $V\in {\bf C}(S)$ is stationary in $M$. 
	
	\begin{claim}
		There exists $V\in {\bf C}(S)$ that is $(G,\mZ_2)$-almost minimizing in annuli (of boundary type) in the sense of \cite[Definition 11]{wang2022min}. 
	\end{claim}
	\begin{proof}[Proof of Claim 2]
	Suppose none of $V\in {\bf C}(S)$ is $(G,\mZ_2)$-almost minimizing in annuli in the sense of \cite[Definition 11]{wang2022min}. 
	Then there is a new sequence $S^*=\{\varphi_i^*\}_{i\in\N}$ of mappings $\varphi_i^*: J(1, l_i)_{0}\rightarrow \Z_{n}^G(M_{\eta_i} ;\mZ_2)$ for some $l_i\geq k^i_{j(i)}\to\infty$ as $i\to\infty$, such that 
	\begin{itemize}
		\item[(i)] ${\bf L}(S^*)<{\bf L}(S) = W^G(M,\partial M)$; 
		\item[(ii)] $\varphi_i$ and $\varphi_i^*$ are $1$-homotopic in $\Z_{n}^G(M_{\eta_i};\mZ_2)$ with $\M$-finenesses tending to zero, \\
			(Specifically, there is a map $\psi_i: I(1,l_i)_0\times J(1,l_i)_0 \to \Z_{n}^G(M_{\eta_i} ;\mZ_2)$ so that $\mf_\M(\psi_i)\to 0$ as $i\to\infty$, $\psi_i([0],\cdot )=\varphi_i\circ {\bf n}_i$, and $\psi_i([1],\cdot )=\varphi_i^*$, where ${\bf n}_i={\bf n}(l_i, k^i_{j(i)} )$);
		\item[(iii)] $\spt(\psi_i(t,x) - \varphi_i\circ {\bf n}_i (x)) \subset\subset M_{a_0}$, for any $t\in I(1,l_i)_0$ and $x\in J(1, l_i)_{0}$;
		\item[(iv)] for any $x\in J(1, l_i)_{0}$, if $\M(\varphi_i\circ {\bf n}_i(x))<W^G(M,\partial M) - \delta/4$, then $\psi_i(\cdot , x) \equiv \varphi_i\circ {\bf n}_i(x)$ is a constant discrete homotopy at $x$. 
	\end{itemize}
	Indeed, since each $V\in {\bf C}(S)$ is supported in $ M_{2a_0}$, we can take $G$-annuli $\{\an(p(V), r_i- s_i, r_i+s_i)\}_{i=1}^{27}$ in $ M_{a_0}$ as in \cite[Theorem 4.10, Part 1]{pitts2014existence}, which implies all the deformations will be restricted in $M_{a_0}$. 
	Using \cite[Theorem 3.14]{wang2022free} and $\dist_{M}(G\cdot p, \cdot)$, we can make the constructions in \cite[Theorem 4.10, Part 2-9]{pitts2014existence} with $G$-invariant objects. 
	Then the rest parts in \cite[Theorem 4.10]{pitts2014existence} are purely combinatorial which would carry over with $M_{a_0}$ in place of $M$. 
	This gives (i)-(iii). 
	Moreover, by taking the constant $\epsilon_2$ in \cite[Theorem 4.10, Part 3]{pitts2014existence} smaller than $\delta /8$, we have $\psi_i(\cdot, x) \equiv \varphi_i \circ {\bf n}_i(x)$ provided $\M(\varphi_i \circ {\bf n}_i(x)) < W^G(M,\partial M) - \delta/4$ (c.f. Part 10(c), Part 14 and 18 in \cite[Theorem 4.10]{pitts2014existence}). 
	
	Next, we can extend $\varphi_i^*$ (for $i$-large) to an $\mF$-continuous map $\widetilde{\Phi}_i^* \in \mathcal{P}^G(M,\partial M)$ so that $\widetilde{\Phi}_i^*\llcorner [0,1/3] = \widetilde{\Phi}_i\llcorner[0,1/3] = \Phi_i\llcorner [0, 1/3]$. 
	Indeed, take any $1$-cell $\alpha = [x_0, x_1]\in J(1, l_i)_1$, we will construct the extension $\widetilde{\Phi}_i^*\llcorner\alpha $ separately in two cases. 
	
	{\bf Case 1: $\max\{\M(\varphi_i\circ {\bf n}_i(x_0)), \M(\varphi_i\circ {\bf n}_i(x_1))\} < W^G(M,\partial M) - \delta /4$.}
	
	By (ii)(iv), we can define $\widetilde{\Phi}_i^* \llcorner \alpha := \widetilde{\Phi}_i\circ f_\alpha$ as the extension of $\varphi_i^*\llcorner\alpha_0$, where $f_\alpha: \alpha= [x_0,x_1]\to [{\bf n}_i(x_0), {\bf n}_i(x_1)]$ is an affine transformation. 
	Hence, in this case, we have 
	\begin{equation}\label{Eq: 1}
		\widetilde{\Phi}_i^*\llcorner\alpha \subset \Z_n^G(M_{\eta_i};\mZ_2),    \quad{\rm and}\quad\widetilde{\Phi}_i^*(x) = \widetilde{\Phi}_i({\bf n}_i(x)) ~ \forall x\in\alpha_0=\{x_0,x_1\}.
	\end{equation}
	In particular, $\widetilde{\Phi}_i^*(1) = \widetilde{\Phi}_i(1)=0$ provided $\mf_\M(\psi_i)< W^G(M,\partial M) - \delta /4$, which holds for $i$-large. 
	Additionally, it follows from (1), Theorem \ref{Thm: interpolation}(i)(iii), and the choice of $\alpha$ that
	\begin{eqnarray}\label{Eq: 1.1}
		\sup\{\M(\widetilde{\Phi}_i^*(x)) : x\in\alpha \} &= & \sup\{\M(\widetilde{\Phi}_i(x)) : x\in f_\alpha(\alpha)\} \nonumber
		\\ &\leq & \sup\{\M(\varphi_i(x)) : x\in \partial f_\alpha(\alpha)\} + C_i\mf_{\M}(\varphi_i)
		\\ &< & W^G(M,\partial M) - \delta /4 + C_i\mf_{\M}(\varphi_i) \nonumber
		\\ &\leq & W^G(M,\partial M) - \delta /5, \quad \mbox{for $i$-large},\nonumber
	\end{eqnarray}
	where $C_i=C_0(M_{\eta_i},G,1)$ is given by Theorem \ref{Thm: interpolation}.

	{\bf Case 2: $\max\{\M(\varphi_i\circ {\bf n}_i(x_0)), \M(\varphi_i\circ {\bf n}_i(x_1))\} \geq W^G(M,\partial M) - \delta /4$.}
	 
	Denote by $A_i\subset J$ the union of all $1$-cells of this case in $J(1, l_i)_1$. 
	Take $i$ sufficiently large so that $\mf_\M(\psi_i)<\delta /4$ (by (ii)). 
	Then $\M(\varphi_i\circ {\bf n}_i(x)) \geq W^G(M,\partial M) - \delta /2$ for all $x\in J(1, l_i)_0\cap A_i$. 
	By (2) and (iii), we have $\varphi_i^*(x)=\psi_i^*([1],x)$ is supported in $ M_{a_0}$ for all $x\in J(1, l_i)_0\cap A_i$. 
	Applying Theorem \ref{Thm: interpolation} to $\varphi_i^*\llcorner [J(1, l_i)_0\cap A_i]$ in $M_{a_0}$ (for $i$-large) will give an $\M$-continuous extension $\widetilde{\Phi}_i^* : A_i\to \Z_n^G(M_{a_0};\mZ_2)$ so that 
	\begin{equation}\label{Eq: 2}
		\sup\{\M(\widetilde{\Phi}_i^*(x)) : x\in A_i\} \leq \sup\{\M(\varphi_i^*(x)): x\in J(1, l_i)_0\cap A_i\} + C_0\mf_{\M}(\psi_i) ,
	\end{equation}
	where $C_0=C_0(M_{a_0},G,1)\geq 1$ is a uniform constant. 
	Note for any $x\in \partial A_i$, we must have $\M(\varphi_i\circ {\bf n}_i(x))< W^G(M,\partial M) - \delta /4$. 
	Hence, by (iv) and Theorem \ref{Thm: interpolation}(i), 
	\begin{equation}\label{Eq: 3}
		 \widetilde{\Phi}_i^*(x) = \varphi_i^*(x)= \varphi_i\circ {\bf n}_i(x)= \widetilde{\Phi}_i({\bf n}_i(x)) \quad \forall x\in \partial A_i . 
	\end{equation}
	
	It now follows from (\ref{Eq: 1})(\ref{Eq: 3}) that $ \widetilde{\Phi}_i^*: I\to \Z_n^G(M;\mZ_2)$ is a well-defined $\mF$-continuous map so that $\widetilde{\Phi}_i^*\llcorner [0,1/3] = \widetilde{\Phi}_i\llcorner[0,1/3] = \Phi_i\llcorner [0, 1/3]$, $\widetilde{\Phi}_i^*(1)=0$, and $\widetilde{\Phi}_i^*\llcorner J\subset \Z_n^G(M_{\eta_i};\mZ_2) $, which implies $\widetilde{\Phi}_i^*\in \mathcal{P}^G(M,\partial M)$. 
	Therefore, by (\ref{Eq: small area near boundary})(\ref{Eq: 1.1})(\ref{Eq: 2})(i)(ii), 
	$$W^G(M,\partial M)\leq {\bf L}(\{ \widetilde{\Phi}_i^*\}_{i\in\N}) \leq \max\{W^G(M,\partial M) - \delta /5,  ~{\bf L}(\{\varphi_i^* \}_{i\in\N}) \} < W^G(M,\partial M),$$
	which is a contradiction. 
	\end{proof}
	
	Thus, there must exist $V\in {\bf C}(S)$ that is $(G,\mZ_2)$-almost minimizing in annuli (of boundary type) in the sense of \cite[Definition 11]{wang2022min}. 
	Since ${\bf C}(S)\subset \V_n^G(M_{2a_0})$, the interior regularity result \cite[Theorem 7]{wang2022min} (modified in \cite[Theorem 4.18]{wang2022equivariant}) indicates that $V$ is an integral $G$-varifold induced by closed smooth embedded minimal $G$-hypersurfaces.
\end{proof}

\section{$G$-sweepouts in positive Ricci curvature $G$-manifolds}\label{Sec: sweepouts}

In this section, we always assume $(M^{n+1}, g_{_M})$ is a closed connected orientable Riemannian manifold with positive Ricci curvature $\Ric_M>0$, and $G$ is a compact Lie group acting on $M$ isometrically so that $3\leq {\rm codim}(G\cdot p)\leq 7$ for all $p\in M$. 
The goal of us is to associate an $\mF$-continuous $G$-sweepout to each closed minimal $G$-hypersurface in $M$.

To begin with, we collect some notations and classical results for minimal hypersurfaces. 
Let $\Sigma\subset M$ be a closed smooth embedded minimal hypersurface. 
Recall the second variation of $\Sigma$ for the area functional is given by 
\begin{eqnarray}\label{Eq: first and second variation formula}
	\delta^2\Sigma(X) := \frac{d^2}{d^2t}\Big|_{t=0}{\rm Area}(F_t(\Sigma)) = - \int_\Sigma \langle L_\Sigma(X^\perp ), X^\perp \rangle ,
\end{eqnarray}
where $L_\Sigma: \mathfrak{X}^\perp(\Sigma)\to \mathfrak{X}^\perp(\Sigma)$ is the Jacobi operator of $\Sigma$, and $\{F_t\}$ are diffeomorphisms generated by $X\in \mathfrak{X}(M)$. 
Then we denote by
\begin{itemize}
	\item ${\rm Index}(\Sigma)$: the {\em Morse index} of $\Sigma$, i.e. the number of the negative eigenvalues (counted with multiplicities) of $L_\Sigma$;
	\item $ \mu_1(\Sigma) $: the first eigenvalue of $L_\Sigma$.
\end{itemize}
If ${\rm Index}(\Sigma)=0 $ or equivalently $\mu_1(\Sigma)\geq 0$, then we say $\Sigma$ is {\em stable}. 

For the case that $\Sigma\subset M$ is a $G$-invariant minimal hypersurface, we have $L_\Sigma(X)\in  \mathfrak{X}^{\perp,G}(\Sigma)$ for all $X\in \mathfrak{X}^{\perp,G}(\Sigma)$, where $ \mathfrak{X}^{\perp,G}(\Sigma)$ is the space of normal $G$-vector fields on $\Sigma$. 
By restricting $L_\Sigma$ to $\mathfrak{X}^{\perp,G}(\Sigma)$, we make the following definition:
\begin{definition}\label{Def: index and G-index}
	Let $\Sigma\subset M$ be a closed smooth embedded minimal $G$-hypersurface. 
	The {\em equivariant Morse index} (or {\em $G$-index} for simplicity) ${\rm Index}_G(\Sigma)$ is defined by the number of the negative eigenvalues (counted with multiplicities) of $L_\Sigma\llcorner \mathfrak{X}^{\perp,G}(\Sigma)$. 
	Additionally, we denote $\mu_1^G(\Sigma)$ as the first eigenvalue of $L_\Sigma\llcorner \mathfrak{X}^{\perp,G}(\Sigma)$. 
\end{definition}

Suppose $\Sigma$ is a closed minimal $G$-hypersurface with a $G$-invariant unit normal $\nu$, and $u_1$ is the first eigenfunction of $L_\Sigma$. 
Then for any $g\in G$, the $G$-invariance of $\Sigma$ and $\nu$ indicates $u_1\circ g$ is also the first eigenfunction of $L_\Sigma$. 
It is well-known that $ \mu_1(\Sigma) $ has multiplicity one and the first eigenfunction $u_1$ does not change sign. 
Hence, $u_1\circ g = u_1$ for all $g\in G$, which suggests $u_1\nu\in\mathfrak{X}^{\perp,G}(\Sigma)$ and the following lemma: 
\begin{lemma}\label{Lem: first eigenfunction}
	If $\Sigma$ is a closed minimal $G$-hypersurface with a $G$-invariant unit normal $\nu$. 
	Then the first eigenfunction $u_1>0$ of $L_\Sigma$ is $G$-invariant and $\mu_1(\Sigma) = \mu_1^G(\Sigma)$. 
\end{lemma}

Since we mainly consider the ambient manifolds with positive Ricci curvature, we collect the following useful results which are well-known to experts (c.f. \cite[Section 2]{zhou2017min}). 
\begin{lemma}\label{Lem: hypersurface in positive Ricci manifold}
	Suppose $(M^{n+1}, g_{_M})$ is a closed connected orientable Riemannian manifold. 
	Let $\Sigma, \Sigma_1,\Sigma_2\subset M$ be closed embedded hypersurfaces. Then we have
	\begin{itemize}
		\item[(i)] if $\Sigma$ is connected, then $\Sigma$ is orientable if and only if it is $2$-sided (i.e. $\Sigma$ has a unit normal vector field); 
		\item[(ii)] if $\Sigma$ is connected and separates $M$, i.e. $M\setminus \Sigma$ has two connected components, then $\Sigma$ is orientable. 
	\end{itemize}
	Moreover, suppose $M$ has positive Ricci curvature. Then we have 
	\begin{itemize} 
		\item[(iii)] if $\Sigma$ is connected and orientable, then $\Sigma$ separates $M$; 
		\item[(iv)] if $\Sigma$ is minimal and $2$-sided, then it can not be stable, i.e. $\mu_1(\Sigma)<0$; 
		\item[(v)] if $\Sigma_1,\Sigma_2$ are minimal hypersurfaces, then $\Sigma_1\cap\Sigma_2\neq\emptyset$;
	\end{itemize}
\end{lemma}

After involving the actions of $G$, a connected component of some $G$-hypersurface $\Sigma$ may not be $G$-invariant. 
Hence, we introduce the following notions of equivariant connectivity. 
\begin{definition}\label{Def: G-connected}
	Let $U\subset M$ be a $G$-invariant subset with connected components $\{U_i\}_{i=1}^m$. 
	Then we say $U$ is {\em $G$-connected} if for any $i,j\in\{1,\cdots,m\}$, there exists $g_{ij}\in G$ so that $g_{ij}\cdot U_j = U_i$. 
	Additionally, we say $U' \subset U$ is a {\em $G$-connected component} (or {\em $G$-component} for simplicity) of $U$, if $U' $ has the form of $\cup_{j=1}^l U_{i(j)}$ and is $G$-connected. 
\end{definition}

Note any $G$-subset $U$ of $M$ can be separated into some $G$-components. 
Additionally, by the above proposition, it is easy to show the following results:
\begin{lemma}\label{Lem: G-hypersurface in positive Ricci manifold}
	Suppose $(M^{n+1}, g_{_M})$ is a closed connected orientable Riemannian manifold with positive Ricci curvature, and $G$ is a compact Lie group acting on $M$ isometrically. 
	Let $\Sigma\subset M$ be a closed embedded minimal $G$-hypersurface. 
	Then $\Sigma$ is connected and 
	\begin{itemize}
		\item if $\Sigma$ has a $G$-invariant unit normal, then $\Sigma$ separates $M$ into two $G$-components. 
		\item if $\Sigma$ does not admit a $G$-invariant unit normal, then $M\setminus\Sigma$ is $G$-connected. 
	\end{itemize}
\end{lemma}
\begin{proof}
	It follows from Lemma \ref{Lem: hypersurface in positive Ricci manifold}(v) that $\Sigma$ is connected. 
	If $\Sigma$ has a $G$-invariant unit normal $\nu$, then by Lemma \ref{Lem: hypersurface in positive Ricci manifold}(i)(iii), $M\setminus \Sigma$ has two connected components $M_1,M_2$, with $\nu$ pointing inward $M_1$. 
	Since $\nu$ and $M_1\cup M_2$ are $G$-invariant, we have $g_*\nu = \nu$ and $g\cdot M_i = M_i$ for all $g\in G$ and $i\in\{1,2\}$, which indicates each $M_i$ is $G$-connected. 
	If the unit normal $\nu$ exists but is not $G$-invariant, then there exists $g\in G$ so that $g_*\nu = -\nu$ pointing inward $M_2$, which implies $g\cdot M_1 = M_2$, and thus $M_1\cup M_2$ is $G$-connected. 
	If $\Sigma$ does not admit a unit normal, then $M\setminus \Sigma$ has only one component which is also $G$-connected. 
\end{proof}

Recall that, in \cite{zhou2015min}, Zhou constructed sweepouts of $M$ by separating orientable and non-orientable minimal hypersurfaces. 
It follows from \cite[Lemma 3.3]{zhou2015min} that the orientability of a connected closed hypersurface is equivalent to the non-connectivity of its unit normal bundle. 
Hence, after involving the actions of $G$, we shall separate the constructions by the $G$-connectivity (Definition \ref{Def: G-connected}) of the unit normal bundle for minimal $G$-hypersurfaces. 

Therefore, we denote 
\begin{equation}\label{Eq: minimal hypersurfaces set}
	\mS^G(M) := \left\{\begin{array}{l|l} \Sigma^n\subset M^{n+1}  & \begin{array}{l} \mbox{$\Sigma$ is a closed smooth embedded} \\ \mbox{minimal $G$-hypersurface in $(M, g_{_M})$} \end{array}\end{array}\right\}. 
\end{equation}
By Theorem \ref{Thm: min-max theorem}, $\mS^G(M)\neq \emptyset$ provided $3\leq {\rm codim}(G\cdot p)\leq 7, \forall p\in M$. 
Define then 
\begin{eqnarray}
	\mS^G_+(M) := \{\Sigma\in \mS^G(M): \mbox{$\Sigma$ has a $G$-invariant unit normal} \}, \nonumber
\end{eqnarray}
and $\mS^G_-(M) := \mS^G(M)\setminus \mS^G_+(M)$. 
It follows directly from Lemma \ref{Lem: G-hypersurface in positive Ricci manifold} that 
$$\Sigma\in \mS^G_-(M) ~\Leftrightarrow ~\mbox{$S\Sigma$ is $G$-connected} ~\Leftrightarrow ~ \mbox{$M\setminus\Sigma$ is $G$-connected} , $$
where $S\Sigma = \{v\in{\bf N}\Sigma : |v|=1 \}$ is the unit normal bundle of $\Sigma$.  

Moreover, for any $\Sigma\in\mS^G_-(M)$, we can cut $M$ along $\Sigma$ to obtain a new manifold $\widetilde{M}$ so that $\widetilde{M}$ is locally isometric to $M$, $G$ acts on $\widetilde{M}$ by isometries, and $\partial\widetilde{M}\in \mS^G_+(\widetilde{M} )$ is a $G$-invariant double cover of $\Sigma$. 
Specifically, let $r >0$ be small enough so that the normal exponential map $\exp^\perp_{\Sigma} : {\bf N}\Sigma \to M$ is a $G$-equivariant diffeomorphism on $B_{2r}(\Sigma):=\{p\in M: \dist_M(\Sigma,p )<2r\}$. 
Hence, we have 
\begin{equation}\label{Eq: double cover map}
	E: S\Sigma \times (-2r, 2r) \to B_r(\Sigma ), \qquad E(v,t) := \exp^\perp_\Sigma(t\cdot v)
\end{equation}
is a double cover of $B_{2r}(\Sigma)$. 
Define the action of $G$ on $S\Sigma \times (-2r, 2r)$ by $g\cdot (v,t):= (g_*v,t)$ for any $v\in S\Sigma$ and $t\in (-2r,2r)$, which indicates $E$ is $G$-equivariant and 
\begin{equation}\label{Eq: double cover Sigma}
	\widetilde{\Sigma} = S\Sigma \times \{0\} 
\end{equation} 
is a $G$-equivariant double cover of $\Sigma$. 
Let $E_r := E\llcorner (S\Sigma \times (r, 2r))$ be a $G$-equivariant diffeomorphism on $B_{2r}(\Sigma)\setminus\Clos(B_r(\Sigma))$. 
Then by gluing $M\setminus \Clos(B_r(\Sigma )$ and $S\Sigma \times [0, 2r)$ on $B_{2r}(\Sigma)\setminus\Clos(B_r(\Sigma))$ with $E_r$, we can define 
\begin{equation}\label{Eq: open up}
	\widetilde{M} := \big(M\setminus \Clos(B_r(\Sigma )) \big)  \cup_{E_r}  \big(S\Sigma \times [0, 2r)\big), 
\end{equation}
as a compact manifold with boundary $\partial \widetilde{M} = \widetilde{\Sigma}$. 
Then we have 
\begin{equation}\label{Eq: reduce to M}
	F: \widetilde{M} \to M, \qquad
	F := \left\{\begin{array}{ll}{ id ,} & {{\rm in~} M\setminus \Clos(B_r(\Sigma )),} \\ {E, }&{{\rm in~} S\Sigma \times [0, 2r),}\end{array}\right.
\end{equation}
is a $G$-equivariant smooth map so that $F\llcorner  (\widetilde{M}\setminus\widetilde{\Sigma})$ gives a diffeomorphism to $M\setminus \Sigma$, and $F\llcorner\widetilde{\Sigma}$ gives a double cover of $\Sigma$. 
Using $F$, we can pull back the metric $g_{_M}$ from $M$ to $\widetilde{M}$ so that $F$ is a local isometry and $G$ acts on $\widetilde{M}$ by isometries. 
Thus, $\widetilde{\Sigma}$ is a minimal $G$-hypersurface in $\widetilde{M}$ with an inward pointing $G$-invariant unit normal. 
In particular, $\Sigma\in\mS^G_-(M)$ implies $S\Sigma$ and $M\setminus\Sigma$ are both $G$-connected, and thus $\widetilde{M}$ is $G$-connected.

\subsection{$G$-sweepouts correspond to $\Sigma\in\mS^G(M)$}

\begin{proposition}\label{Prop: sweepout construction 1}
	Given any $\Sigma\in \mS^G_+(M)$, there exists an $\mF $-continuous $G$-sweepout $\Phi: [-1,1]\to \Z_n^G(M;\mZ_2)$ of $M$ so that 
	\begin{itemize}
		\item[(i)] $\Phi(0)=[[\Sigma]]$, $\Phi(-1)=\Phi(1)=0$;
		\item[(ii)] $\M(\Phi(x))\leq {\rm Area}(\Sigma)$ with equality only if $x=0$. 
	\end{itemize}
\end{proposition}
\begin{proof}
	By Lemma \ref{Lem: G-hypersurface in positive Ricci manifold}, $M\setminus \Sigma$ has two $G$-components $M_1$ and $M_2$ so that the unit normal $\nu$ of $\Sigma$ pointing inward $M_1$. 
	Additionally, it follows from Lemma \ref{Lem: first eigenfunction} and \ref{Lem: hypersurface in positive Ricci manifold} that the first eigenfunction $u_1>0$ of $L_\Sigma$ is a $G$-invariant function satisfying $L_\Sigma u_1 = - \mu_1 (\Sigma) u_1>0$. 
	
	Denote by $d_{ \pm}$ the signed distance function to $\Sigma$ so that $d_{ \pm} = \dist_M(\Sigma,\cdot)$ in $M_1$, and $d_{ \pm} = -\dist_M(\Sigma,\cdot)$ in $M_2$. 
	Let $X\in \mathfrak{X}^G(M)$ be a $G$-vector field with $X=(u_1\circ n_\Sigma ) \cdot \nabla d_{\pm}$ in a neighborhood of $\Sigma$, where $n_\Sigma$ is the nearest projection (in $M$) to $\Sigma$. 
	Then we consider the $G$-equivariant variation $\{\Sigma_t := F_t(\Sigma)\}_{t\in [-r, r ]}$ of $\Sigma$, where $\{F_t\}$ are the $G$-equivariant diffeomorphisms generated by $X$. 
	By the second variation formula (\ref{Eq: first and second variation formula}), we have 
	$$ \delta^2\Sigma(X) = \frac{d^2}{dt^2}\Big|_{t=0}{\rm Area}(\Sigma_t) =- \int_\Sigma u_1L_\Sigma u_1 <0, 
		\quad \frac{d}{dt}\Big|_{t=0}\langle \vec{H}_{\Sigma_t}, \nabla d_{\pm}\rangle = L_\Sigma u_1 >0,$$
	where $\vec{H}_{\Sigma_t}$ is the mean curvature vector field of $\Sigma_t$. 
	Thus, for $r>0$ small enough, 
	$${\rm Area}(\Sigma_t) < {\rm Area}(\Sigma) \quad {\rm and}
		\quad \langle \vec{H}_{\Sigma_t}, \nabla \dist_{M}(\Sigma, \cdot)\rangle>0,\quad {\rm for~all~} t\in [-r,0)\cup (0,r].$$ 
	Define $\Phi(x) := [[\Sigma_x]] = (F_x)_\#[[\Sigma]]\in\Z_n^G(M;\mZ_2)$ for $x\in [-r,r]$, which is $\mF$-continuous. 
	
	Since $u_1>0$, $\{\Sigma_t\}_{t\in [-r,r]}$ is a smooth foliation of a $G$-neighborhood of $\Sigma$, and $\Sigma_t\subset M_1$ for $t>0$, $\Sigma_t\subset M_2$ for $t<0$. 
	We now consider the compact manifolds $M_{1}' := M_1\setminus \{\Sigma_t\}_{t\in [0,r)}$ and $M_2':=M_2\setminus \{\Sigma_t\}_{t\in (-r,0]}$, whose boundary $\partial M_i' = \Sigma_{r_i}$, ($i\in\{1,2\}$, $r_1=r$, $r_2=-r$), is a $G$-hypersurface with positive mean curvature pointing inward $M_i'$. 
	
	Suppose $W^G(M_i', \partial M_i') > {\rm Area}(\Sigma_{r_i})$. 
	Then by Theorem \ref{Thm: min-max theorem for M with boundary}, there exists a closed minimal $G$-hypersurface $\Sigma'$ in the interior of $M_i'$. 
	Noting $\Sigma\cap \Sigma' = \emptyset$, we get a contradiction from Lemma \ref{Lem: hypersurface in positive Ricci manifold}(v).
	Therefore, $W^G(M_i', \partial M_i')\leq {\rm Area}(\Sigma_{r_i})$. 
	By Definition \ref{Def: sweepout and width in M with boundary} and Corollary \ref{Cor: F-continuous width}, there exist $\epsilon>0$ small enough and an $\mF$-continuous $G$-sweepout $\Phi_i:[0,1]\to \Z_n^G(M_i;\mZ_2)$ so that $\Phi_i(0)= [[\Sigma_{r_i}]]$, $\Phi_i(1)=0$, and 
	$$\sup\{\M(\Phi_i(x)):x\in[0,1]\}\leq W^G(M_i', \partial M_i')+\epsilon \leq {\rm Area}(\Sigma_{r_i})+\epsilon < {\rm Area}(\Sigma). $$ 
	Finally, by reparametrization, we have a well-defined map $\Phi:[-1,1]\to \Z_n^G(M;\mZ_2)$, 
	$$ \Phi(x) := \left\{\begin{array}{ll}
				{ \Phi_2( -\frac{3}{2} x - \frac{1}{2} ) ,} & {x\in[-1,-\frac{1}{3} ] ,} 
				\\ {(F_{3rx})_\#[[\Sigma]] ,} & {x\in [ -\frac{1}{3}, \frac{1}{3}],}
				\\ { \Phi_1( \frac{3}{2} x - \frac{1}{2} ) ,} & {x\in[\frac{1}{3} , 1 ] ,} 
				\end{array}\right.  $$
	continuous in the $\mF$-topology satisfying (i) and (ii). 
	Additionally, the arguments before Definition \ref{Def: sweepout in closed manifolds} and \ref{Def: sweepout and width in M with boundary} indicate $F_M(\Phi) = [[M_{2}]] + [[M_1]] = [[M]]$, where $F_M$ is given by (\ref{Eq: isomorphism map}). 
	Hence, we have $\Phi\in\mathcal{P}^G(M)$. 
\end{proof}

\begin{proposition}\label{Prop: sweepout construction 2}
	Given any $\Sigma\in \mS^G_-(M)$, there exists an $\F $-continuous $G$-sweepout $\Phi: [0,1]\to \Z_n^G(M;\mZ_2)$ of $M$ with no concentration of mass on orbits so that 
	\begin{itemize}
		\item[(i)] $\Phi(0)=\Phi(1)=0$;
		\item[(ii)] $\sup \{\M(\Phi(x)) : x\in [0,1]\} < 2{\rm Area}(\Sigma)$.
	\end{itemize}
\end{proposition}
\begin{proof}
	Let $\widetilde{\Sigma} = S\Sigma \times \{0\}$ and $\widetilde{M}$ be given by (\ref{Eq: double cover Sigma})(\ref{Eq: open up}). 
	Then $\widetilde{M}$ is $G$-connected, ${\rm Area}(\widetilde{\Sigma})=2{\rm Area}(\Sigma)$, and $\widetilde{\Sigma}$ has a $G$-invariant unit normal $\tilde{\nu}$ pointing inward $\widetilde{M}$. 
	Let $\tau: \widetilde{\Sigma} \to \widetilde{\Sigma}$ be the isometric involution, i.e. $\tau(v,0)=(-v,0)$ for $v\in S\Sigma$. 
	
	Using the constructions in Proposition \ref{Prop: sweepout construction 1} with $\widetilde{M}$ in place of $M_1$, we get an $\mF$-continuous $G$-sweepout $\widetilde{\Phi}:[0,1]\to \Z_n^G(\widetilde{M};\mZ_2)$ so that $\widetilde{\Phi}(0) = [[\widetilde{\Sigma}]]$, $\widetilde{\Phi}(1) =0$, and $\M(\widetilde{\Phi}(x) ) \leq 2{\rm Area}(\Sigma)$ for all $x\in [0,1]$ with equality only at $x=0$. 
	Additionally, for $t\in [0, 1/3]$, $\widetilde{\Phi}(t) = \widetilde{\Sigma}_t := [[\exp_{\widetilde{\Sigma}}^\perp(t\tilde{u} \tilde{\nu}) ]]$, where $\tilde{u}=3tr\tilde{u}_1$ and $\tilde{u}_1:\widetilde{\Sigma} \to \R^+$ is the $G$-invariant first eigenfunction of $L_{\widetilde{\Sigma}}$ with eigenvalue $\mu_1(\widetilde{\Sigma}) = \mu_1^G(\widetilde{\Sigma}) <0$.

	Now, by the second variation formula (\ref{Eq: first and second variation formula}), there are $\delta_0\in (0, 1/3)$ and $C_0>0$ so that 
	\begin{eqnarray}\label{Eq: construct sweepout 2(1)}
		\M(\widetilde{\Phi}(t)) =  \mH^n( \widetilde{\Sigma}_t ) &=&  \mH^n( \widetilde{\Sigma}) - \frac{t^2}{2} \int_{\widetilde{\Sigma}}\langle L_{\widetilde{\Sigma}} \tilde{u} \tilde{\nu} , \tilde{u} \tilde{\nu} \rangle + O(t^3) \nonumber
		\\
		&\leq & \mH^n( \widetilde{\Sigma}) - C_0t^2 
	\end{eqnarray}
	for all $t\in (0,\delta_0)$. 
	For any $\delta\in (0, \delta_0)$ (will be specified later), the $\mF$-continuity of $\widetilde{\Phi}$ and Proposition \ref{Prop: sweepout construction 1}(ii) suggest the existence of $\epsilon>0$ with
	\begin{equation}\label{Eq: construct sweepout 2(2)}
		\M(\widetilde{\Phi}(t) )\leq \mH^n( \widetilde{\Sigma}) - \epsilon , \qquad\forall t\in [\delta ,1 ].
	\end{equation}
	Now, we will open up $\widetilde{\Sigma}_t$, $t\in [0,\delta]$, at some orbit to decrease the area.

	Specifically, let $G\cdot \tilde{p} \subset \widetilde{\Sigma}^{prin}$ be any principal orbit of $\widetilde{\Sigma}$. 
	Then by the $G$-invariance of $\tilde{\nu}$ and \cite[Corollary 2.2.2]{berndt2016submanifolds}, $G\cdot \tilde{p}\subset \widetilde{M}^{prin}$ is also a principal orbit in $\widetilde{M}$. 
	Note either $G\cdot \tilde{p} = G\cdot \tau(\tilde{p})$ or $G\cdot \tilde{p}\cap G\cdot \tau(\tilde{p})=\emptyset$. 
	Thus, we can define $P:=G\cdot \tilde{p} \cup G\cdot \tau(\tilde{p})$ as a $G$-invariant submanifold in $\widetilde{\Sigma}$ with dimension $n-l$. 
	By assumptions, $3\leq {\rm codim}(G\cdot \tilde{p}) = l+1\leq 7$. 
	
	{\bf Case 1: $3\leq l\leq 6$.} 
	For any $r>0,t \in [0,\delta ]$, we define the following $G$-invariant sets:
	\begin{itemize}
		\item $\tB_r(P):=\{\tilde{q}\in\widetilde{\Sigma} : \dist_{\widetilde{\Sigma}}(\tilde{q}, P ) < r \} \subset \widetilde{\Sigma}$ ;
		\item $\tB_{r,t}(P) := \{ \exp_{\widetilde{\Sigma}}^\perp\big((t \tilde{u} \tilde{\nu} )(\tilde{q}) \big) :  \tilde{q}\in\tB_r(P) \}\subset \widetilde{\Sigma}_t$;
		\item $\tC_{r,t}(P) := \{ \exp_{\widetilde{\Sigma}}^\perp\big((s \tilde{u} \tilde{\nu} )(\tilde{q}) \big) :  \tilde{q}\in\tB_r(P), s\in[0,t]\}$. 
	\end{itemize}
	Then for $R,\delta>0$ small enough, it follows from the integral formula in \cite[(C.4)]{wang2022free} that 
	\begin{equation}\label{Eq: catenoid estimate 1}
		ct r^{l-1}  \leq \mathcal{H}^{n}(\tC_{r, t}(P)) \leq Ct r^{l-1} 
		\quad {\rm and}\quad 
		c r^{l}  \leq \mathcal{H}^{n}(\tB_{r, t}(P)) \leq C r^{l},
	\end{equation} 
	for all $r\in [0,R], t\in [0,\delta ]$, where $c,C>0$ are constants depending on $\widetilde{\Sigma},\widetilde{M},P$. 
	Define 
	$$ \widetilde{\Sigma}_{r,t} := \left( \widetilde{\Sigma}_t\setminus \tB_{r,t}(P)\right) \cup \tC_{r,t}(P) \cup \tB_r(P), \quad r\in [0,R], t\in [0,\delta ].$$
	By (\ref{Eq: construct sweepout 2(1)})(\ref{Eq: catenoid estimate 1}), $\|\widetilde{\Sigma}_{r,t}\|(\widetilde{M}\setminus\widetilde{\Sigma} ) \leq \mH^n(\widetilde{\Sigma} ) - C_0t^2 - cr^l + Ctr^{l-1}$. 
	
	Note $Ctr^{l-1}\leq \frac{C_0}{2}t^2 + \frac{2C}{C_0}r^{2l-2}$ and $l\geq 3$ in this case. 
	We can take $R>0$ small enough so that $\frac{2C}{C_0}R^{l-2} < \frac{c}{2}$. 
	%
	Hence,
	$$ \|\widetilde{\Sigma}_{r,t}\|(\widetilde{M}\setminus\widetilde{\Sigma} ) \leq \mH^n(\widetilde{\Sigma} ) - \frac{C_0}{2} t^2 - \frac{c}{2}r^l ,$$
	for all $t\in [0,\delta ]$, $r\in [0,R]$, and 
	$$\widetilde{\Sigma}_{t}' := 
				\left\{\begin{array}{ll}
					{ \widetilde{\Sigma}_{R,2t} } & {~  t\in [0, \frac{\delta}{2} ],} 
				\\ { \widetilde{\Sigma}_{2R(1- t/\delta) ,\delta} } & {~ t\in [ \frac{\delta}{2} ,\delta ],} 
				\end{array}\right. 
		\quad  
		\|\widetilde{\Sigma}_{t}' \|(\widetilde{M}\setminus\widetilde{\Sigma} )\leq 
				\left\{\begin{array}{ll}
					{ \mH^n(\widetilde{\Sigma} ) - cR^l /2} & {~  t\in [0, \frac{\delta}{2} ],} 
				\\ { \mH^n(\widetilde{\Sigma} ) - C_0 \delta^2/2 } & {~ t\in [ \frac{\delta}{2} ,\delta ].} 
				\end{array}\right. 
	$$
	Set $\epsilon' :=\min\{\epsilon, cR^l/2, C_0\delta^2/2 \}$ and define $\widetilde{\Phi}'(t) :=[[ \widetilde{\Sigma}_t']]$ for $t\in [0,\delta]$ in this case. 
	
	{\bf Case 2: $l=2$.} 
	For $R>r>0$ small enough, let $\eta_{r,R}:\widetilde{\Sigma}\to [0,1]$ be the $G$-invariant logarithmic cut-off function defined by 
	$$ \eta_{r,R}(\tilde{q} ):=\left\{\begin{array}{ll}
					{ 1 } & {\quad  \tilde{q}\notin \tB_R(P) ,} 
				\\ { (\log r - \log(\dist_{\widetilde{\Sigma}}(\tilde{q}, P ) ) )/ (\log r- \log R) } & {\quad \tilde{q}\in \tB_R(P)\setminus \tB_r(P),} 
				\\ { 0 } & {\quad  \tilde{q}\in \tB_r(P) ,} 
				\end{array}\right.   $$
	which is also $\tau$-invariant. 
	Consider 
	$$ \widetilde{\Sigma}_{r,R,t} := \exp_{\widetilde{\Sigma}}^\perp(t \eta_{r,R}\tilde{u}\tilde{\nu} ) .$$
	By \cite[Proposition 2.5]{ketover2020catenoid} and \cite[(C.4)]{wang2022free}, we can take $R,\delta>0$ small enough so that
	\begin{eqnarray*}
		\|\widetilde{\Sigma}_{r,R,t}\|(\widetilde{M}\setminus\widetilde{\Sigma} ) &\leq & \mH^n(\widetilde{\Sigma}\setminus \tB_r(P) ) + \frac{t^2}{2}\int_{\widetilde{\Sigma}} |\nabla (\eta_{r,R}\tilde{u}) |^2 - (\Ric(\tilde{\nu},\tilde{\nu} ) + |A|^2 ) (\eta_{r,R}\tilde{u})^2 
		\\ && + Ct^3\int_{\widetilde{\Sigma}} 1 +|\nabla (\eta_{r,R}\tilde{u}) |^2
		\\ &\leq & \mH^n(\widetilde{\Sigma}\setminus \tB_r(P) ) - C_1t^2  + C_2t^2\int_{\widetilde{\Sigma}} |\nabla \eta_{r,R} |^2 + t^2\int_{\tB_R(P)} \tilde{u} \eta_{r,R} \nabla \tilde{u}\nabla \eta_{r,R}
		\\ && + Ct^3\int_{\widetilde{\Sigma}} 1 + 2\eta_{r,R}^2|\nabla  \tilde{u}|^2 + 2\tilde{u}^2|\nabla \eta_{r,R} |^2 
		\\ &\leq & \mH^n(\widetilde{\Sigma}) - cr^2-C_1t^2 + \frac{C_3}{\log(R/r)} t^2 + C_4R^2t^2 + C_5t^3 + \frac{C_6}{\log(R/r)} t^3,
	\end{eqnarray*}
	for all $r\in (0,R), t\in[0,\delta ]$, where $c,C,C_i>0$ are uniform constants depending on $\widetilde{\Sigma},\widetilde{M},P$. 
	Set $R,\delta >0$ even smaller so that $C_4R^2< C_1/4$, $C_5\delta <C_1/4$, and $C_6\delta <C_3$. 
	Then choose $r>0$ small enough with $\frac{2C_3}{\log(R/r)} < \frac{C_1}{4} $. 
	Thus, 
	$$\|\widetilde{\Sigma}_{r,R,t}\|(\widetilde{M}\setminus\widetilde{\Sigma} )\leq  \mH^n(\widetilde{\Sigma}) - cr^2 - \frac{C_1}{2} t^2 + \frac{2C_3}{\log(R/r)}t^2\leq \mH^n(\widetilde{\Sigma}) - cr^2 - \frac{C_1}{4} t^2 $$
	for all $t\in [0,\delta]$, and 
	$$\widetilde{\Sigma}_{t}' := 
				\left\{\begin{array}{ll}
					{ \widetilde{\Sigma}_{r,R,2t} } & {~  t\in [0, \frac{\delta}{2} ],} 
				\\ { \widetilde{\Sigma}_{2r(1- t/\delta),2R(1- t/\delta) ,\delta} } & {~ t\in [ \frac{\delta}{2} ,\delta ],} 
				\end{array}\right. 
		\quad  
		\|\widetilde{\Sigma}_{t}' \|(\widetilde{M}\setminus\widetilde{\Sigma} )\leq 
				\left\{\begin{array}{ll}
					{ \mH^n(\widetilde{\Sigma} ) - cr^2 } & {~  t\in [0, \frac{\delta}{2} ],} 
				\\ { \mH^n(\widetilde{\Sigma} ) - C_1 \delta^2/4 } & {~ t\in [ \frac{\delta}{2} ,\delta ].} 
				\end{array}\right. 
	$$
	In this case, set $\epsilon' :=\min\{\epsilon, cr^2, C_1\delta^2/4 \}$ and define $\widetilde{\Phi}'(t) :=[[ \widetilde{\Sigma}_t']]$ for $t\in [0,\delta]$. 	
	
	In both cases, we define $\widetilde{\Phi}'\llcorner [\delta, 1]=\widetilde{\Phi}\llcorner[\delta, 1 ]$ and see
	\begin{equation}\label{Eq: mass in interior strictly decrease}
		\sup \{ \|\widetilde{\Phi}'(t) \| (\widetilde{M}\setminus\widetilde{\Sigma}) : t\in [0,1]\} \leq \mH^n(\widetilde{\Sigma} )-\epsilon' .
	\end{equation}
	Additionally, by (\ref{Eq: F-metric}), $\widetilde{\Phi}'$ is still an $\mF$-continuous map with $\widetilde{\Phi}'=[[\widetilde{\Sigma}]] $, $\widetilde{\Phi}'(1) =0$.

	Finally, we define $\Phi(x):= F_\#\widetilde{\Phi}'(x)$ for all $x\in [0,1]$, where $F:\widetilde{M}\to M$ is the equivariant local isometry given by (\ref{Eq: reduce to M}). 
	Because $F:\widetilde{M}\setminus\widetilde{\Sigma}\to M\setminus\Sigma$ is an equivariant isometry, the arguments before Definition \ref{Def: sweepout in closed manifolds} and \ref{Def: sweepout and width in M with boundary} indicate $F_M(\Phi) = F_\#([[\widetilde{M}]] ) = M$, where $F_M$ is given by (\ref{Eq: isomorphism map}). 
	Additionally, note $F:\widetilde{\Sigma}\to\Sigma$ is a double cover and $\widetilde{\Sigma}'_t\cap\widetilde{\Sigma}$ is $\tau$-invariant in both cases. 
	Hence, by $\mZ_2$-coefficients and (\ref{Eq: mass in interior strictly decrease}), we have $\Phi(0)=F_\#[[\widetilde{\Sigma}]]=0$ and 
	$$ \M(\Phi(x)) =  \|\widetilde{\Phi}'(t) \| (\widetilde{M}\setminus\widetilde{\Sigma}) \leq \mH^n(\widetilde{\Sigma} )-\epsilon' = 2{\rm Area}(\Sigma) - \epsilon'.$$
	At last, since $\|\Phi(x)\|(B^G_r(p)) \leq \|\widetilde{\Phi}'(x)\|(F^{-1}(B^G_r(p))) \leq 2{\bf m}^G(\Phi', r) $ for all $x\in [0,1], p\in M$, we see ${\bf m}^G(\Phi, r) \leq 2{\bf m}^G(\widetilde{\Phi}', r)$ and $\Phi$ has no concentration of mass on orbits. 
\end{proof}

\section{Proof of the main theorems}\label{Sec: proof main}
Let $\mS^G(M)$ be given in (\ref{Eq: minimal hypersurfaces set}). 
Then we define 
\begin{equation}\label{Eq: least area}
	\mathcal{A}^G(M) := \inf_{\Sigma\in\mS^G(M)}
				\left\{\begin{array}{ll}
					{ {\rm Area}(\Sigma) ,} & {\quad {\rm if}~\Sigma\in\mS^G_+(M)} 
				\\ { 2{\rm Area}(\Sigma) ,} & {\quad {\rm if}~\Sigma\in\mS^G_-(M)} 
				\end{array}\right\}.
\end{equation}

\begin{theorem}\label{Thm: main 1}
	Let $(M^{n+1}, g_{_M})$ be a closed connected orientable Riemannian manifold with positive Ricci curvature, and $G$ be a compact Lie group acting on $M$ isometrically so that $3\leq {\rm codim}(G\cdot p)\leq 7$ for all $p\in M$. 
	Then the equivariant min-max hypersurface $\Sigma$ corresponding to the fundamental class $[M]$ is a connected minimal $G$-hypersurface of multiplicity one with a $G$-invariant unit normal vector field so that
	$$ {\rm Index}_G(\Sigma) = 1 \quad {\rm and}\quad {\rm Area}(\Sigma)=W^G(M)=\A^G(M) .$$
\end{theorem}
\begin{proof}
	By the min-max theorem \ref{Thm: min-max theorem}, there exists an integral $G$-varifold $V\in\V^G_n(M)$ induced by a smooth embedded closed minimal $G$-hypersurface $\Sigma\in\mS^G(M)$ so that $\|V\|(M)=W^G(M) $. 
	Since $M$ has positive Ricci curvature, Lemma \ref{Lem: hypersurface in positive Ricci manifold}(v) indicates that $\Sigma$ is connected, and thus $V=m|\Sigma|$ for some $m\in\{1,2,\dots\}$. 
	Suppose $\Sigma\in\mS^G_-(M)$, then it follows from the last statement in Theorem \ref{Thm: min-max theorem} that $m$ must be even, so $m\geq 2$. 
	However, we have a contradiction $W^G(M)<2{\rm Area}(\Sigma)\leq \|V\|(M) = W^G(M) $ by Proposition \ref{Prop: sweepout construction 2}.
	Therefore, $\Sigma\in\mS^G_+(M)$.
	By Proposition \ref{Prop: sweepout construction 1}, we see
	$W^G(M)\leq {\rm Area}(\Sigma) \leq \|V\|(M) = W^G(M) $, and thus $m=1$. 
	Additionally, by the definition of $\A^G(M)$ and Proposition \ref{Prop: sweepout construction 1}, \ref{Prop: sweepout construction 2}, 
	$$\A^G(M)\leq {\rm Area} (\Sigma) = \|V\|(M) = W^G(M) \leq \A^G(M).$$
	
	Now, it is sufficient to show ${\rm Index}_G(\Sigma) =1$. 
	Suppose ${\rm Index}_G(\Sigma) \geq 2$, and $u_1,u_2$ are the first two $L^2$-orthonormal $G$-invariant eigenfunctions of $L_\Sigma\llcorner \mathfrak{X}^{\perp,G}(\Sigma )$ with negative eigenvalues. 
	Let $u_2\nu$ be a $G$-invariant normal vector field on $\Sigma$ which extends to a smooth vector field $X\in\mathfrak{X}(M)$. 
	Then $X_2:= \int_G (g^{-1})_*X d\mu(g)\in\mathfrak{X}^G(M)$ gives an equivariant extension of $u_2\nu$. 
	Consider the equivariant diffeomorphisms $\{F_s^2\}$ generated by $X_G$ and define $\Phi_s(t):=(F_s^2)_\#\Phi(t)$ for $t\in [-1,1]$, where $\Phi\in\mathcal{P}^G(M)$ is the $\mF$-continuous sweepout given by Proposition \ref{Prop: sweepout construction 1}. 
	Recall that in the proof of Proposition \ref{Prop: sweepout construction 1}, $\Phi(t)= [[\Sigma_t]]=[[F_t^1(\Sigma)]]$ for $t\in [-1/3,1/3]$, where $\{F_t^1\}$ are the equivariant diffeomorphism generated by $X_1\in\mathfrak{X}^G(M) $ with $X_1\llcorner\Sigma = 3ru_1\nu$ for some $r>0$. 
	Hence, for the smooth family $\{F^2_s(\Sigma_t) \}_{s\in[-\sigma ,\sigma ],t\in[-1/3,1/3]}$, the area function $A(s,t):={\rm Area}(F^2_s(\Sigma_t)) = \M(\Phi_s(t))$ satisfies:
	\begin{itemize}
		\item $\nabla A(0,0)=0$ since $\Sigma$ is minimal;
		\item $\frac{\partial^2}{\partial t^2}A(0,0) = -9r^2\int_\Sigma u_1L_\Sigma u_1 <0$ and $\frac{\partial^2}{\partial s^2}A(0,0) = -\int_\Sigma u_2L_\Sigma u_2 <0$;
		\item $\frac{\partial^2}{\partial s \partial t}A(0,0) = -3r\int_\Sigma u_2L_\Sigma u_1 =3r\mu_1(\Sigma)\int_\Sigma u_1u_2=0$. 
	\end{itemize}
	Therefore, we can set $\sigma,\delta >0$ sufficiently small so that 
	$$ \M(\Phi_s(t))= {\rm Area}(F^2_s(\Sigma_t )) < {\rm Area}(\Sigma),\quad\forall t\in[-\delta,\delta ], s\in (0,\sigma ]. $$
	Moreover, there exists $\epsilon>0$ so that $\M(\Phi(t))\leq {\rm Area}(\Sigma) - \epsilon$ for all $t\in [-1,-\delta ]\cup[\delta, 1 ]$ by Proposition \ref{Prop: sweepout construction 1}(ii). 
	Hence, by setting $\sigma>0$ even smaller, we have $\M(\Phi_\sigma(t))=\M((F^2_\sigma)_\#\Phi(t) ) < {\rm Area}(\Sigma)$ for all $t\in [-1,1]$. 
	Note $\Phi_\sigma$ is an $\mF$-continuous curve homotopic to $\Phi$ in $\Z_n^G(M;\mZ_2)$. 
	Thus, 
	$$ W^G(M)\leq \sup \{\M(\Phi_\sigma(t)): t\in [-1,1]\}< {\rm Area}(\Sigma)= W^G(M),$$
	which is a contradiction. So we have ${\rm Index}_G(\Sigma)=1$. 
\end{proof}

As an application, we use the conformal volume to show a genus bound for the equivariant min-max minimal $G$-hypersurface $\Sigma$ in Theorem \ref{Thm: main 1} provided that $\dim(M)=3$ and the actions of $G$ are orientation preserving. 

\begin{theorem}\label{Thm: main 2}
	Let $(M^3, g_{_M})$ be a closed connected oriented Riemannian $3$-manifold with positive Ricci curvature, and $G$ be a finite group acting on $M$ by orientation preserving isometries. 
	Then the equivariant min-max hypersurface $\Sigma$ corresponding to the fundamental class $[M]$ is a connected closed minimal $G$-surface of multiplicity one satisfying 
	$${\rm genus}(\Sigma )\leq 4 K \quad {\rm and} \quad W^G(M)={\rm Area}(\Sigma) \leq \frac{8\pi K}{\inf_{|v|=1}\Ric_M(v,v)},$$
	where $K:=\max_{p\in M} \# G\cdot p \leq \#G$ is the number of points in a principal orbit of $M$. 
	Additionally, $\pi(\Sigma)=\Sigma/G$ is an orientable surface with finite cone singular points of order $\{n_i\}_{i=1}^k$ (i.e. locally modeled by $\mathbb{B}^2_1(0)$ quotient a cyclic rotation group $\mZ_{n_i}$), so that
	$$ \sum_{i=1}^k (1-\frac{1}{n_i} ) \leq 4, \quad {\rm and }\quad  {\rm genus}(\pi(\Sigma))\leq 3.$$
	In particular, if $\Sigma\subset M^{prin} $, i.e. $k=0$, then ${\rm genus}(\Sigma)\leq 1+2K$. 
\end{theorem}
\begin{proof}
	By Theorem \ref{Thm: main 1}, $\Sigma$  is a closed embedded connected minimal $G$-surface with a $G$-invariant unit normal $\nu$ so that ${\rm Area}(\Sigma) = W^G(M)$ and ${\rm Index}_G(\Sigma)=1$. 
	By Lemma \ref{Lem: hypersurface in positive Ricci manifold}, $\Sigma$ has an induced orientation. 
	Additionally, since the unit normal $\nu$ is $G$-invariant, the actions of $G$ on $\Sigma$ are also orientation preserving. 
	Therefore, the orbifold $\underline{ \Sigma}$ induced by $(\Sigma, G)$ is an orientable closed $2$-orbifold whose underlying space is the quotient distance space $(\pi(\Sigma), \dist_{\Sigma/G}) $. 
	
	Let $\Sigma^{prin}$ be the union of principal orbits for the $G$-action on $\Sigma$, and $\underline{ \Sigma}^{prin}$ be the orbifold induced by $(\Sigma^{prin},G)$. 
	Note an orbit $G\cdot p$ is principal in $\Sigma$ (or $M$) if and only if the slice representation of $G_p$ on ${\bf N}_p^\Sigma G\cdot p$ (or ${\bf N}_pG\cdot p$) is trivial (c.f. \cite[Corollary 2.2.2]{berndt2016submanifolds}). 
	Hence, by the $G$-invariance of $\nu$, we see $\Sigma^{prin}\subset M^{prin}$ and thus $K=\# G\cdot p = \#G\cdot q$ for all $p\in \Sigma^{prin}$ and $q\in M^{prin}$. 
	Additionally, it follows from \cite[Chapter IV, Theorem 3.3]{bredon1972introduction} that there is an induced Riemannian metric $g_{_{\underline{\Sigma}}}$ on $\underline{\Sigma}^{prin}$ so that $\pi: \Sigma^{prin}\to \underline{\Sigma}^{prin}$ is an Riemannian submersion. 
	Moreover, since $G$ acts on $\Sigma$ by orientation preserving isometries, the singular points $\underline{\Sigma} \setminus\underline{\Sigma}^{prin} $ are a finite number of cone points $\{[p_i]\}_{i=1}^k$ of orders $n_1,\dots, n_k$. 
	By the orbifold version of Gauss-Bonnet theorem (c.f. \cite[Proposition 2.17]{cooper2000three}), we have 
	\begin{equation}\label{Eq: Gauss-Bonnet}
		\int_{\pi(\Sigma)} K_{\underline{\Sigma}} ~dA_{g_{_{\underline{\Sigma}}}} = 2\pi(\chi(\underline{\Sigma} ) ) = 2\pi\Big( 2-2{\rm genus}(\pi(\Sigma)) - \sum_{i=1}^k(1-\frac{1}{n_i} ) \Big),
	\end{equation}
	where $K_{\underline{\Sigma}} $ is the Gauss curvature of $(\underline{\Sigma}^{prin}, g_{_{\underline{\Sigma}}} )$, and the integral is taken over $\underline{\Sigma}^{prin}$. 
	
	For any $r>0$ small enough, let $\Sigma_r:=\Sigma\setminus \cup_{i=1}^k B_r^G(p_i )$, and $\eta_r $ be the $G$-invariant logarithmic cut-off function on $\Sigma$ given by
	$$ \eta_{r}(p):=
			\left\{\begin{array}{lr}
						0, & d(p) \in [0, r] 
					\\ 2-\frac{2 \log d(p)}{\log r}, & d(p)\in (r,\sqrt{r}] 
					\\ 1, & d(p)\in (\sqrt{r}, \infty)
			 \end{array}\right. ,$$
	where $d(p) := \dist_\Sigma(p, \Sigma\setminus\Sigma^{prin} ) = \dist_\Sigma(p, \cup_{i=1}^k G\cdot p_i)$. 
	Define then $\underline{\Sigma}_r := \underline{\Sigma}\setminus \cup_{i=1}^kB_r([p_i] )$. 
	Note $(\underline{\Sigma}_r, g_{_{\underline{\Sigma}}})$ is a smooth Riemannian manifold (with boundary). 
	We can take any conformal immersion $\phi: \underline{\Sigma}_r \to \mathbb{S}^m$, $m\geq 2$, and define $P:{\rm Conf}(\mathbb{S}^m)\to \mathbb{B}^{m+1}_1(0)$ by
	$$P(h) :=\frac{1}{\int_\Sigma \eta_r u_1}\Big(\int_{\Sigma}(\eta_r u_1)(h_1\circ \phi\circ \pi), \dots, \int_{\Sigma}(\eta_r u_1)(h_{m+1}\circ \phi\circ \pi) \Big),  $$
	where $h=(h_1,\cdots,h_{m+1})\in {\rm Conf}(\mathbb{S}^m)$ is any conformal diffeomorphism of $\mathbb{S}^m$ (under the standard metric), and $u_1:\Sigma\to\R_+$ is the first ($G$-invariant) eigenfunction of $L_\Sigma$. 
	Since $u_1>0$ and $\sum_{j=1}^{m+1} h_j^2=1$, one easily verifies $P$ is well-defined. 
	Meanwhile, for each $x\in \mathbb{B}^{m+1}$, define a conformal diffeomorphism $h_x\in {\rm Conf}(\mathbb{S}^m)$ as in \cite[(1.1)]{montiel1986minimal} by 
	$$h_x(y)=\frac{y+(\mu\langle x, y\rangle+\lambda) x}{\lambda(\langle x, y\rangle+1)}, $$
	where $\lambda = (1-|x|^2)^{-1/2} $ and $\mu=(\lambda -1)|x|^{-2}$. 
	Then we have a continuous map $f: \mathbb{B}^{m+1}_1(0) \to \mathbb{B}^{m+1}_1(0)$ given by $f(x)= P(h_x)$, which can be continuously extended to $\partial \mathbb{B}^{m+1}_1(0) = \mathbb{S}^{m}$ by the identity map. 
	Note $\Clos(\mathbb{B}^{m+1}_1(0)) $ is homotopic to $f(\Clos(\mathbb{B}^{m+1}_1(0)))$, and $\Clos(\mathbb{B}^{m+1}_1(0)) \setminus \{x\}$ is homotopic to $\mathbb{S}^{m}$ for any $x\in \mathbb{B}^{m+1}_1(0)$. 
	Hence, we must have $f$ is surjective. 
	In particular, there exists $h=(h_1,\dots, h_{m+1})\in {\rm Conf}(\mathbb{S}^m)$ so that $P(h)=0$. 
	Thus, $\{\tilde{h}_j := h_j\circ\phi\circ\pi\}_{j=1}^{m+1}$ are $G$-invariant smooth functions on $\Sigma_r$ so that 
	$$ \sum_{j=1}^{m+1}\tilde{h}_j^2=1\quad {\rm and} \quad \int_\Sigma u_1\cdot (\eta_r\tilde{h}_j) = 0, ~ \forall j=1,\dots,m+1.$$
	
	Since ${\rm Index}_G(\Sigma)=1$, we see $\delta^2\Sigma(\eta_r\tilde{h}_j \nu )\geq 0$ for all $j=1,\dots,m+1$, and 
	\begin{eqnarray*}
		\int_{\Sigma_{\sqrt{r}}} \Ric_M(\nu,\nu) + |A|^2 &\leq & \int_\Sigma (\Ric_M(\nu,\nu) + |A|^2)\eta_r^2 = \int_\Sigma (\Ric_M(\nu,\nu) + |A|^2)\sum_{j=1}^{m+1}(\eta_r \tilde{h}_j)^2
		\\ &\leq & \int_\Sigma \sum_{j=1}^{m+1} |\nabla (\eta_r \tilde{h}_j) |^2
		\\ &\leq & \int_\Sigma \sum_{j=1}^{m+1} \Big[ (1+\epsilon)|\nabla \tilde{h}_j |^2 \eta_r^2 + (1+ \frac{1}{\epsilon} )|\nabla \eta_r  |^2\tilde{h}_j^2\Big]
		\\ &\leq & (1+\epsilon)K \cdot \int_{\underline{\Sigma}_r}\sum_{j=1}^{m+1} |\nabla h_j\circ\phi |^2 + (1+ \frac{1}{\epsilon} ) \int_\Sigma |\nabla \eta_r  |^2
		\\ &=& 2(1+\epsilon) K \cdot {\rm Area}(\underline{\Sigma}_r; (h\circ\phi)^*g_{_{\mathbb{S}^{m+1}}} ) + (1+ \frac{1}{\epsilon} ) \int_\Sigma |\nabla \eta_r  |^2,
	\end{eqnarray*}
	where $\epsilon>0$ is any constant, ${\rm Area}(\underline{\Sigma}_r; (h\circ\phi)^*g_{_{\mathbb{S}^{m+1}}} )$ is the area of $\underline{\Sigma}_r$ under the conformal metric $(h\circ\phi)^*g_{_{\mathbb{S}^{m+1}}}$, and the co-area formula is used in the last inequality. 
	Let $A_c(m,\underline{\Sigma}_r )$ be the {\em $m$-conformal area} of $\Sigma$ defined as in \cite{li1982new}: 
	$$ A_c(m,\underline{\Sigma}_r ) := \inf_\phi \sup_{h\in{\rm Conf}(\mathbb{S}^{m}) } {\rm Area}(\underline{\Sigma}_r; (h\circ\phi)^*g_{_{\mathbb{S}^{m}}} ),  $$
	where the infimum is taken over all non-degenerated conformal map $\phi$ of $\underline{\Sigma}_r$ into $\mathbb{S}^{m}$. 
	Since $\phi: \underline{\Sigma}_r \to \mathbb{S}^m$ is arbitrary conformal immersion in the above computation, we have 
	$$\int_{\Sigma_{\sqrt{r}}} \Ric_M(\nu,\nu) + |A|^2  \leq 2(1+\epsilon )K \cdot A_c(m,\underline{\Sigma}_r ) + (1+ \frac{1}{\epsilon} ) \int_\Sigma |\nabla \eta_r  |^2  . $$
	By \cite[Chapter IV, Remark 5.5.1]{hartshorne1977algebraic}, every closed orientable surface can be conformally branched over $\mathbb{S}^2$ with degree $\lfloor ({\rm genus}+3)/2\rfloor$, where $\lfloor a\rfloor$ is the integer part of $a\in\R_+$. 
	It then follows from \cite[Fact 1, 5]{li1982new} that $A_c(m,\underline{\Sigma}_r ) \leq 4\pi \lfloor\frac{{\rm genus}(\pi(\Sigma))+3}{2}\rfloor$, and thus 
	$$ \int_{\Sigma_{\sqrt{r}}} \Ric_M(\nu,\nu) + |A|^2  \leq 4\pi (1+\epsilon )K \cdot 2\Big\lfloor\frac{{\rm genus}(\pi(\Sigma))+3}{2}\Big\rfloor + (1+ \frac{1}{\epsilon} ) \int_\Sigma |\nabla \eta_r  |^2  . $$
	Since $\int_\Sigma |\nabla \eta_r  |^2\to 0$ as $r\to 0$, we first take $r\to 0$ and then let $\epsilon \to 0$, which gives
	$$ \int_{\Sigma} \Ric_M(\nu,\nu) + |A|^2 \leq 4\pi K \cdot 2\Big\lfloor\frac{{\rm genus}(\pi(\Sigma))+3}{2}\Big\rfloor. $$
	Denote by $\{e_i\}_{i=1}^2$ a local orthonormal basis on $\Sigma$. 
	Since $\Ric_M>0$, 
	we have
	$$\Ric_M(\nu,\nu) + |A|^2 = \sum_{i=1}^2\Ric_M(e_i,e_i) -2K_\Sigma > -2K_\Sigma $$ 
	on $\Sigma^{prin}$, where $K_\Sigma$ is the Gauss curvature of $\Sigma$. 
	Therefore, by the co-area formula,
	$$ -2K \int_{\underline{\Sigma}} K_{\underline{\Sigma}} = -2\int_\Sigma K_\Sigma < \int_{\Sigma} \Ric_M(\nu,\nu) + |A|^2 \leq 4\pi K \cdot 2\Big\lfloor\frac{{\rm genus}(\pi(\Sigma))+3}{2}\Big\rfloor . $$
	Then, it follows from the above strict inequality and the Gauss-Bonnet formula (\ref{Eq: Gauss-Bonnet}) that ${\rm genus}(\pi(\Sigma)) \leq 3$, $\sum_{i=1}^k(1-\frac{1}{n_i} )\leq 4$, ${\rm genus}(\pi(\Sigma)) +\sum_{i=1}^k(1-\frac{1}{n_i} ) < 5$, and 
	$$ {\rm genus}(\Sigma ) = 1+ K \Big[ {\rm genus}(\pi(\Sigma)) - 1 + \sum_{i=1}^k(1-\frac{1}{n_i} ) \Big] <1+ 4K. $$ 
	In particular, if $\Sigma\subset M^{prin} $, then $\sum_{i=1}^k(1-\frac{1}{n_i} )=0$ and ${\rm genus}(\Sigma)=1+K({\rm genus}(\pi(\Sigma)) -1) \leq 1+2K$. 
	Finally, we see
	\begin{eqnarray*}
		2c_M W^G(M) &\leq & \int_\Sigma \sum_{i=1}^2\Ric_M(e_i,e_i) \\
		&\leq & 4\pi K\cdot 2\Big\lfloor\frac{{\rm genus}(\pi(\Sigma))+3}{2}\Big\rfloor + 2K \int_{\underline{\Sigma}} K_{\underline{\Sigma}} \\
		&= & 4\pi K\cdot \Big(2- 2{\rm genus}(\pi(\Sigma)) - \sum_{i=1}^k(1-\frac{1}{n_i} ) + 2\Big\lfloor\frac{{\rm genus}(\pi(\Sigma))+3}{2}\Big\rfloor \Big)\\ 
		&\leq & 16\pi K,
	\end{eqnarray*}
	where $\Ric_M \geq c_M >0 $. 
\end{proof}




\bibliographystyle{abbrv}

\providecommand{\bysame}{\leavevmode\hbox to3em{\hrulefill}\thinspace}
\providecommand{\MR}{\relax\ifhmode\unskip\space\fi MR }
\providecommand{\MRhref}[2]{%
  \href{http://www.ams.org/mathscinet-getitem?mr=#1}{#2}}
\providecommand{\href}[2]{#2}

\bibliography{reference}   

\end{document}